\documentclass[a4paper,reqno]{amsart}
      \usepackage{etex}
      \usepackage[OT2, T1]{fontenc}
      \usepackage[utf8]{inputenc}
      \usepackage{url}
      \usepackage{amsmath}
      \usepackage{mathtools}
      \usepackage{lmodern}
      \usepackage{array}
      \usepackage{graphicx}
      \usepackage{amsfonts}
      \usepackage{amssymb}
      \usepackage{extpfeil}
\usepackage{soul}
\usepackage{xargs}
      \usepackage{xcolor}
      \definecolor{royalblue}{RGB}{64, 106, 212}
      \definecolor{link}{RGB}{11,0,128}
      \definecolor{olivegreen}{RGB}{128, 128, 0}
      \usepackage{aliascnt}
      \usepackage[colorlinks=true, citecolor=royalblue,linkcolor=black,urlcolor=link]{hyperref}
      \usepackage{amsthm}
      \usepackage{esint}
      \usepackage{epic}
      \usepackage{rotating}

      \usepackage{calligra}
      \usepackage{xkeyval}

      \usepackage{exscale, relsize}
      \usepackage{stackengine}
      \usepackage{scalerel}
      \usepackage{setspace}
      \usepackage{paralist}
      \usepackage{colonequals}		
      \usepackage{mathabx}		
      \usepackage[left]{lineno}
      \usepackage{amscd}
      \usepackage[all,cmtip]{xy}	
      \usepackage{sseq}			
      \usepackage{verbatim}
      \usepackage{parskip}
      \usepackage{microtype}
      \usepackage{ragged2e}
            \usepackage[in]{fullpage}
      \usepackage{graphicx}
      \usepackage{mathrsfs} 			
      \usepackage{bm} 				
      \usepackage{cleveref}
      \usepackage{enumitem}
      \usepackage{moreenum}			
      \usepackage{blindtext}
      \usepackage{tikz-cd}
      \usepackage{tikz}
      \usepackage{subfiles}
      \usetikzlibrary{shapes.geometric}
      \usetikzlibrary{automata, bending,matrix,arrows,decorations.pathmorphing,backgrounds,positioning,fit,petri,arrows.meta}
      \usepackage{textcomp}
      \usepackage{indentfirst}
      \usetikzlibrary{arrows}
      \usepackage{enumitem}
      \tikzset{commutative diagrams/.cd,arrow style=tikz,diagrams={>=latex'}}
      \usepackage{filecontents}
      \usepackage[alphabetic,lite]{amsrefs} 	
      \usepackage{stmaryrd}	
      \usepackage{subfiles}
      \usepackage{ragged2e}
      \usetikzlibrary{arrows}
      \usepackage{stackengine}

\usepackage[colorinlistoftodos,prependcaption,textsize=tiny]{todonotes}
\newcommandx{\unsure}[2][1=]{\todo[linecolor=red,backgroundcolor=red!25,bordercolor=red,#1]{#2}}
\newcommandx{\change}[2][1=]{\todo[linecolor=blue,backgroundcolor=blue!25,bordercolor=blue,#1]{#2}}
\newcommandx{\info}[2][1=]{\todo[linecolor=OliveGreen,backgroundcolor=OliveGreen!25,bordercolor=OliveGreen,#1]{#2}}
\newcommandx{\improvement}[2][1=]{\todo[linecolor=Plum,backgroundcolor=Plum!25,bordercolor=Plum,#1]{#2}}
\newcommandx{\thiswillnotshow}[2][1=]{\todo[disable,#1]{#2}}

\usepackage{marginnote}
\newcommand{\mytodo}[2][]{{%
 \let\marginpar\marginnote
 \reversemarginpar
 \renewcommand{\baselinestretch}{0.8}%
 \todo[#1]{#2}}}

         \newcommand{\bA}{\mathbb{A}}

         \newcommand{\cU}{\mathcal{U}}

         \newcommand{\fp}{\mathfrak{p}}

         \newcommand{\eps}{\epsilon}

         \renewcommand{\b}{\textbf}

         \providecommand{\fps}[1]{[\![#1]\!]}
         \providecommand{\lps}[1]{(\!(#1)\!)}

         \providecommand{\fps}[1]{\llbracket#1\rrbracket}
         \providecommand{\lps}[1]{(\!(#1)\!)}

\newextarrow{\xbigtoto}{{15}{15}{15}{12}}
   {\bigRelbar\bigRelbar{\bigtwoarrowsleft\rightarrow\rightarrow}}
         
         \DeclareMathOperator{\Spec}{Spec}		                       
         \DeclareMathOperator{\Proj}{Proj}		                       
         \DeclareMathOperator{\Ass}{Ass}			                       
         \DeclareMathOperator{\Supp}{Supp}		                       
         \DeclareMathOperator{\length}{length}		                                                  

         \newcommand{\ba}{\begin{aligned}}
         \newcommand{\ea}{\end{aligned}}
         \newcommand{\be}{\begin{equation}}
         \newcommand{\ee}{\end{equation}}
         \newcommand{\pf}{\begin{proof}}
         \newcommand{\bpf}{\begin{proof}}
         \newcommand{\epf}{\end{proof}}
         \newcommand{\bsol}{\begin{solution}}
         \newcommand{\esol}{\end{solution}}
         \newcommand{\bthm}{\begin{thm}}
         \newcommand{\ethm}{\end{thm}}

         \newcommand{\bprop}{\begin{prop}}
         \newcommand{\eprop}{\end{prop}}
         \newcommand{\bcor}{\begin{cor}}
         \newcommand{\ecor}{\end{cor}}

         \newcommand{\brem}{\begin{rem}}
         \newcommand{\erem}{\end{rem}}

         \newcommand{\brems}{\begin{rems} \hfill \begin{enumerate}[label=\b{\thenumberingbase.},ref=\thenumberingbase]}
         
         \newcommand{\erems}{\end{enumerate} \end{rems}}
         \newcommand{\begs}{\begin{egs} \hfill \begin{enumerate}[label=\b{\thenumberingbase.},ref=\thenumberingbase]}
         
         \newcommand{\eegs}{\end{enumerate} \end{egs}}
         \newcommand{\eremstweak}{\end{enumerate} \end{rems-tweak}}
         \newcommand{\eremst}{\end{enumerate} \end{rems-tweak}}
         \newcommand{\blem}{\begin{lemma}}
         \newcommand{\elem}{\end{lemma}}

         \newcommand{\bconj}{\begin{conj}}
         \newcommand{\econj}{\end{conj}}
         \newcommand{\bprob}{\begin{Problem}}
         \newcommand{\eprob}{\end{Problem}}

         \newcommand{\bq}{\begin{Q}}
         \newcommand{\eq}{\end{Q}}
         \newcommand{\benum}{\begin{enumerate}[label={{\upshape(\alph*)}}]}
         \newcommand{\benuma}{\begin{enumerate}[label={{\upshape(\arabic*)}}]}
         \newcommand{\benumr}{\begin{enumerate}[label={{\upshape(\roman*)}}]}
         \newcommand{\eenum}{\end{enumerate}}
         \newcommand{\bc}{}
         \newcommand{\bd}{\begin{defn}}
         \newcommand{\ed}{\end{defn}}
         \newcommand{\bque}{\begin{que}}
         \newcommand{\eque}{\end{que}}
         \newcommand{\bfct}{\begin{fact}}
         \newcommand{\efct}{\end{fact}}

         \newcommand{\beg}{\begin{eg}}
         \newcommand{\eeg}{\end{eg}}

         \newcommand{\bcl}{\begin{claim}}
         \newcommand{\ecl}{\end{claim}}

         \newcommand{\Inf}{\mathrm{inf}}

\tikzset{
    labl/.style={anchor=south, rotate=90, inner sep=.5mm}
}

\newaliascnt{numberingbase}{subsection}
\numberwithin{equation}{numberingbase}

\newcommand{\statementskip}{9pt plus 2pt minus 1pt}
\newcommand{\subsectionbeforeskip}{9pt plus 2pt minus 1pt}
\newcommand{\subsectionafterskip}{4pt plus 1pt minus .5pt}
\makeatletter
\def\subsection{\@startsection{subsection}{2}%
  \z@{\subsectionbeforeskip}{\subsectionafterskip}%
  {\normalfont\bfseries}}
\def\subsubsection{\@startsection{subsubsection}{3}%
  \z@{7pt plus 1.5pt minus 1pt}{3pt plus 1pt minus .5pt}%
  {\normalfont\itshape}}
\makeatother

\newtheoremstyle{thms}{\statementskip}{\statementskip}{\itshape}{}{\bfseries}{.}{ }{}
\theoremstyle{thms}
\newaliascnt{conj}{numberingbase}
\newtheorem{conj}[conj]{Conjecture}
\aliascntresetthe{conj}
\newaliascnt{corollary}{numberingbase}
\newtheorem{corollary}[corollary]{Corollary}
\aliascntresetthe{corollary}
\newaliascnt{cor}{numberingbase}
\newtheorem{cor}[cor]{Corollary}
\aliascntresetthe{cor}
\newaliascnt{lemma}{numberingbase}
\newtheorem{lemma}[lemma]{Lemma}
\aliascntresetthe{lemma}
\newaliascnt{sublemma}{equation}

\aliascntresetthe{sublemma}
\newaliascnt{prop}{numberingbase}
\newtheorem{prop}[prop]{Proposition}
\aliascntresetthe{prop}
\newaliascnt{proposition}{numberingbase}
\newtheorem{proposition}[proposition]{Proposition}
\aliascntresetthe{proposition}
\newaliascnt{Q}{numberingbase}
\newtheorem{Q}[Q]{Question}
\aliascntresetthe{Q}
\newaliascnt{thm}{numberingbase}
\newtheorem{thm}[thm]{Theorem}
\aliascntresetthe{thm}
\newaliascnt{theorem}{numberingbase}
\newtheorem{theorem}[theorem]{Theorem}
\aliascntresetthe{theorem}
\newaliascnt{variant}{numberingbase}
\newtheorem{variant}[variant]{Variant}
\aliascntresetthe{variant}

\newtheoremstyle{claims}{\statementskip}{\statementskip}{}{}{\itshape}{.}{ }{}
\theoremstyle{claims}
\newaliascnt{claim}{equation}
\newtheorem{claim}[claim]{Claim}
\aliascntresetthe{claim}
\newaliascnt{cl-tweak}{subsubsection}
\newtheorem{cl-tweak}[cl-tweak]{Claim}
\aliascntresetthe{cl-tweak}
\Crefname{cl-tweak}{Claim}{Claims}

\newtheoremstyle{defs}{\statementskip}{\statementskip}{}{}{\bfseries}{.}{ }{}
\theoremstyle{defs}
\newaliascnt{defn}{numberingbase}
\newtheorem{defn}[defn]{Definition}
\aliascntresetthe{defn}
\newaliascnt{definition}{numberingbase}
\newtheorem{definition}[definition]{Definition}
\aliascntresetthe{definition}
\newaliascnt{eg}{numberingbase}
\newtheorem{eg}[eg]{Example}
\aliascntresetthe{eg}
\newaliascnt{example}{numberingbase}
\newtheorem{example}[example]{Example}
\aliascntresetthe{example}
\newtheorem*{egs}{Examples}
\newaliascnt{rem}{numberingbase}
\newtheorem{rem}[rem]{Remark}
\aliascntresetthe{rem}
\newaliascnt{remark}{numberingbase}
\newtheorem{remark}[remark]{Remark}
\aliascntresetthe{remark}
\newtheorem*{rems}{Remarks}

\Crefname{claim}{Claim}{Claims}
\Crefname{sublemma}{Lemma}{Lemmas}
\Crefname{conj}{Conjecture}{Conjectures}
\Crefname{cor}{Corollary}{Corollaries}
\Crefname{defn}{Definition}{Definitions}
\Crefname{eg}{Example}{Examples}
\Crefname{prop}{Proposition}{Propositions}
\Crefname{Q}{Question}{Questions}
\Crefname{rem}{Remark}{Remarks}
\Crefname{thm}{Theorem}{Theorems}
\Crefname{variant}{Variant}{Variants}
\crefname{claim}{claim}{claims}
\crefname{sublemma}{lemma}{lemmas}
\crefname{conj}{conjecture}{conjectures}
\crefname{cor}{corollary}{corollaries}
\crefname{corollary}{corollary}{corollaries}
\crefname{defn}{definition}{definitions}
\crefname{definition}{definition}{definitions}
\crefname{eg}{example}{examples}
\crefname{example}{example}{examples}
\crefname{lemma}{lemma}{lemmas}
\crefname{prop}{proposition}{propositions}
\crefname{proposition}{proposition}{propositions}
\crefname{Q}{question}{questions}
\crefname{rem}{remark}{remarks}
\crefname{remark}{remark}{remarks}
\crefname{thm}{theorem}{theorems}
\crefname{theorem}{theorem}{theorems}
\crefname{variant}{variant}{variants}
\Crefname{corollary}{Corollary}{Corollaries}
\Crefname{definition}{Definition}{Definitions}
\Crefname{example}{Example}{Examples}
\Crefname{lemma}{Lemma}{Lemmas}
\Crefname{proposition}{Proposition}{Propositions}
\Crefname{remark}{Remark}{Remarks}
\Crefname{theorem}{Theorem}{Theorems}

\theoremstyle{thms}
\newaliascnt{thm-tweak}{subsection}
\newtheorem{thm-tweak}[thm-tweak]{Theorem}
\aliascntresetthe{thm-tweak}
\crefname{thm-tweak}{theorem}{theorems}
\Crefname{thm-tweak}{Theorem}{Theorems}
\newaliascnt{lemma-tweak}{subsection}
\newtheorem{lemma-tweak}[lemma-tweak]{Lemma}
\aliascntresetthe{lemma-tweak}
\crefname{lemma-tweak}{lemma}{lemmas}
\Crefname{lemma-tweak}{Lemma}{Lemmas}
\newaliascnt{cor-tweak}{subsection}
\newtheorem{cor-tweak}[cor-tweak]{Corollary}
\aliascntresetthe{cor-tweak}
\crefname{cor-tweak}{corollary}{corollaries}
\Crefname{cor-tweak}{Corollary}{Corollaries}
\newaliascnt{prop-tweak}{subsection}
\newtheorem{prop-tweak}[prop-tweak]{Proposition}
\aliascntresetthe{prop-tweak}
\crefname{prop-tweak}{proposition}{propositions}
\Crefname{prop-tweak}{Proposition}{Propositions}
\newaliascnt{conj-tweak}{subsection}
\newtheorem{conj-tweak}[conj-tweak]{Conjecture}
\aliascntresetthe{conj-tweak}
\crefname{conj-tweak}{conjecture}{conjectures}
\Crefname{conj-tweak}{Conjecture}{Conjectures}

\theoremstyle{defs}
\newaliascnt{defn-tweak}{subsection}
\newtheorem{defn-tweak}[defn-tweak]{Definition}
\aliascntresetthe{defn-tweak}
\crefname{defn-tweak}{definition}{definitions}
\Crefname{defn-tweak}{Definition}{Definitions}
\newaliascnt{eg-tweak}{subsection}
\newtheorem{eg-tweak}[eg-tweak]{Example}
\aliascntresetthe{eg-tweak}
\crefname{eg-tweak}{example}{examples}
\Crefname{eg-tweak}{Example}{Examples}
\newtheorem*{rems-tweak}{Remarks}
\newaliascnt{rem-tweak}{subsection}
\newtheorem{rem-tweak}[rem-tweak]{Remark}
\aliascntresetthe{rem-tweak}
\crefname{rem-tweak}{remark}{remarks}
\Crefname{rem-tweak}{Remark}{Remarks}

\creflabelformat{equation}{(#2\textcolor{red}{#1}#3)}
\creflabelformat{theorem}{#2\textcolor{red}{#1}#3}
\creflabelformat{lemma}{#2\textcolor{red}{#1}#3}
\creflabelformat{proposition}{#2\textcolor{red}{#1}#3}
\creflabelformat{corollary}{#2\textcolor{red}{#1}#3}
\creflabelformat{remark}{#2\textcolor{red}{#1}#3}
\creflabelformat{definition}{#2\textcolor{red}{#1}#3}
\creflabelformat{example}{#2\textcolor{red}{#1}#3}
\creflabelformat{claim}{#2\textcolor{red}{#1}#3}
\creflabelformat{conjecture}{#2\textcolor{red}{#1}#3}
\creflabelformat{cl-tweak}{#2\textcolor{red}{#1}#3}
\creflabelformat{sublemma}{#2\textcolor{red}{#1}#3}
\creflabelformat{lemma-tweak}{#2\textcolor{red}{#1}#3}
\creflabelformat{conj}{#2\textcolor{red}{#1}#3}
\creflabelformat{conj-tweak}{#2\textcolor{red}{#1}#3}
\creflabelformat{cor}{#2\textcolor{red}{#1}#3}
\creflabelformat{cor-tweak}{#2\textcolor{red}{#1}#3}
\creflabelformat{defn}{#2\textcolor{red}{#1}#3}
\creflabelformat{defn-tweak}{#2\textcolor{red}{#1}#3}
\creflabelformat{eg}{#2\textcolor{red}{#1}#3}
\creflabelformat{eg-tweak}{#2\textcolor{red}{#1}#3}
\creflabelformat{prop}{#2\textcolor{red}{#1}#3}
\creflabelformat{prop-tweak}{#2\textcolor{red}{#1}#3}
\creflabelformat{Q}{#2\textcolor{red}{#1}#3}
\creflabelformat{rem}{#2\textcolor{red}{#1}#3}
\creflabelformat{rem-tweak}{#2\textcolor{red}{#1}#3}
\creflabelformat{thm}{#2\textcolor{red}{#1}#3}
\creflabelformat{thm-tweak}{#2\textcolor{red}{#1}#3}
\creflabelformat{variant}{#2\textcolor{red}{#1}#3}

\newtheoremstyle{subsection-tweak}
   {\subsectionbeforeskip}
   {\subsectionafterskip}%
   {}
   {}%
   {\bfseries}
   {}%
   {\newline}
   {\thmnumber{\@{#1}{}\@{#2}.}%
    \thmnote{~{\bfseries#3.}}}

\makeatletter
\newcommand{\ppafterheading}{%
  \@ifnextchar\label{\ppafterheadinglabel}{\ignorespaces}%
}
\def\ppafterheadinglabel\label#1{\label{#1}\ignorespaces}
\makeatother

\newenvironment{pp}[1][]{%
  \par\addvspace{\subsectionbeforeskip}%
  \refstepcounter{subsection}%
  \noindent\textbf{\thesubsection.\if\relax\detokenize{#1}\relax\else\ #1.\fi}%
  \par\nobreak\vspace{\subsectionafterskip}\noindent\ppafterheading
}{%
  \par\addvspace{\subsectionafterskip}%
}
\newcommand{\bpp}{\begin{pp}}
\newcommand{\epp}{\end{pp}}
\newaliascnt{pp-t}{subsubsection}
\newtheorem{pp-t}[pp-t]{}
\aliascntresetthe{pp-t}
\crefname{pp-t}{section}{sections}
\Crefname{pp-t}{Section}{Sections}

\theoremstyle{subsection-tweak}
\newaliascnt{conventions}{subsection}

\aliascntresetthe{conventions}
\crefname{conventions}{section}{sections}
\Crefname{conventions}{Section}{Sections}

\theoremstyle{subsection-tweak}
\newtheorem{pp-tweak}{}
\crefname{pp-tweak}{section}{sections}
\Crefname{pp-tweak}{Section}{Sections}

      \makeatletter
      \def\@tocline#1#2#3#4#5#6#7{
          \begingroup
          \@ifempty{#4}{}{}

          \parindent\z@ \leftskip#3\relax \advance\leftskip\@tempdima\relax
          #5\hskip-\@tempdima
            \ifcase #1
             \or\or \hskip 2em \or \hskip 1em \else \hskip 3em \fi%
            #6\nobreak\relax
          \dotfill\hbox to\@pnumwidth{\@tocpagenum{#7}}\par
          \nobreak
          \endgroup
        }
       \def\l@section{\@tocline{1}{0pt}{1pc}{}{}}

      \renewcommand{\tocsection}[3]{%
        \indentlabel{\@ifnotempty{#2}{\makebox[2em][l]{%
          \ignorespaces#1 \bfseries{#2}.\hfill}}}\bfseries{#3}
          \vspace{-3.5pt}}

      \renewcommand{\tocsubsection}[3]{%
        \indentlabel{\@ifnotempty{#2}{\hspace*{-0.5em}\makebox[3em][l]{%
          \ignorespaces#1#2.\hfill}}}#3
          \vspace{-4.5pt}}

\makeatother

\makeatletter 
\newcommand\appendix@section[1]{%
  \refstepcounter{section}%
  \orig@section*{Appendix \@Alph\c@section. #1}%
}
\let\orig@section\section
\g@addto@macro\appendix{\let\section\appendix@section}
\makeatother

      \DeclareMathVersion{normal2}

\providecommand{\m}{\mathfrak m}
\providecommand{\n}{\mathfrak n}
\renewcommand{\length}{\ell}
\renewcommand{\eps}{\varepsilon}
\providecommand{\catlim}{\operatorname{catlim}}
\renewcommand{\Inf}{\operatorname{Inf}}
\providecommand{\Min}{\operatorname{Min}}

\providecommand{\Sing}{\operatorname{Sing}}
\providecommand{\Spec}{\operatorname{Spec}}
\providecommand{\Supp}{\operatorname{Supp}}
\providecommand{\Ass}{\operatorname{Ass}}
\providecommand{\cinf}{c_{\infty}}
\providecommand{\cLM}{c_{\mathrm{LM}}}
\providecommand{\muo}{\mu^{\circ}}
\providecommand{\resultref}[1]{%
  \hyperref[#1]{\textcolor{red}{\ref*{#1}}}}
\makeatletter
\DeclareRobustCommand{\eqref}[1]{\textup{\tagform@{\resultref{#1}}}}
\makeatother

\hypersetup{
  pdftitle={Asymptotic Lech Bounds and Bertini Flags},
  pdfauthor={Hong-Ge Chen and Fei Liu},
  pdfsubject={Commutative algebra and Hilbert--Samuel multiplicity},
  pdfkeywords={Lech's inequality, arbitrary characteristic,
    singular locus, Bertini flag, Hilbert--Samuel multiplicity,
    catapower, mixed multiplicity}
}

\title[Asymptotic Lech bounds]
{Asymptotic Lech Bounds and Bertini Flags}

\author{Hong-Ge Chen}
\address{Hong-Ge Chen, School of Mathematics and Statistics,
Central China Normal University, Wuhan 430079, China}
\email{hongge\_chen@ccnu.edu.cn}

\author{Fei Liu}
\address{Fei Liu, Department of Mathematics, Run Run Shaw Building,
The University of Hong Kong, Hong Kong}
\email{liufei54@pku.edu.cn}

\def\UTFviii@defined#1{%
  \ifx#1\relax
      \PackageError{inputenc}{Unicode\space char\space\expandafter
                              \UTFviii@splitcsname\string#1\relax
                              \MessageBreak
                              not\space set\space up\space
                              for\space use\space with\space LaTeX}\@eha
  \else\expandafter
    #1%
  \fi
}

\makeatletter
\def\UTFviii@defined#1{%
  \ifx#1\relax
      ?%
  \else\expandafter
    #1%
  \fi
}

\makeatother


\begin{document}

\begin{abstract}
For a $d$-dimensional Noetherian local ring $(R,\m)$, with
$d>0$, let $\cinf(R)$ be the limit, as $N\to\infty$, of the
suprema of $e_R(I)/(d!\length_R(R/I))$ over ideals $I$ with
$\sqrt I=\m$ and $\length_R(R/I)>N$.
We prove that if $R$ is complete, reduced, equidimensional,
and Cohen--Macaulay on the punctured spectrum, with $d\geq2$
and infinite residue field, then every nonzerodivisor $x\in\m$
satisfies
\[
 \cinf(R)\leq
 \max\left\{\cinf(R/xR),\sup_{\fp\ne\m}e(R_{\fp})\right\}.
\]
As an application, we prove the asymptotic Lech conjecture of
Huneke, Ma, Quy, and Smirnov in arbitrary characteristic:
$\cinf(R)=1$ whenever $\widehat R$ has isolated singularities.
This also settles the conjecture of Ma and Smirnov that the
Lech--Mumford constant $\cLM(R)$ is attained when
$\widehat R$ has isolated singularities and $\cLM(R)>1$.
For reduced equidimensional Cohen--Macaulay standard graded
algebras of dimension $\geq2$ over infinite perfect fields,
we prove attainment and rationality of $\cLM$ for the completed
vertex local rings of sufficiently high Veronese subrings.
In characteristic zero, a refinement of our method produces
nonisolated klt affine cones in every dimension $d\geq3$,
with $d-2$-dimensional singular loci, whose completed vertex
local rings have $\cinf=1$ and arbitrarily large $\cLM$.
\end{abstract}

\maketitle

\hypersetup{linkcolor=black}

\section{Introduction}

\subsection{Asymptotic Lech's inequality}

Lech \cite{Lech} studied the relation between the
Hilbert--Samuel multiplicity of an ideal and its colength.  He proved
that for a Noetherian local ring $(R,\m)$ of dimension $d>0$
\[
 e_R(I)\leq d!e(R)\length_R(R/I),
\]
for every $\m$-primary ideal $I\subset R$.  Here
$e_R(I)$ denotes the Hilbert--Samuel multiplicity of $I$, and
$e(R)=e_R(\m)$.
Multiplicity--colength ratios also enter Mumford's local
numerical conditions for asymptotic GIT semistability
\cite{MumfordStability}*{\S3}.

The asymptotic problem asks when the bound for regular local
rings remains valid, up to an arbitrarily small error, for possibly singular
rings: can the factor $e(R)$ be replaced by $1+\eps$, for every
$\eps>0$, uniformly over all $\m$-primary ideals of sufficiently
large colength?

Define the \emph{asymptotic Lech constant} by
\begin{equation}\label{eq:universal-tail-bounds}
 1\leq\cinf(R):=\lim_{N\to\infty}
 \sup_{\substack{\sqrt I=\m\\\length_R(R/I)>N}}
 \frac{e_R(I)}{d!\length_R(R/I)}
 \leq e(R).
\end{equation}
The lower bound follows by considering powers of a fixed
$\m$-primary ideal, and the upper bound from Lech's inequality.
Thus $\cinf(R)=1$ for a regular local ring.
Huneke, Ma, Quy, and Smirnov
\cite{HMQS}*{Conjecture~1.2(a)} conjectured that the same equality
holds whenever $\widehat R$ has isolated singularities.

They proved the conjecture in positive characteristic with
perfect residue field, using Hilbert--Kunz multiplicity
\cite{HMQS}*{Theorem~A(1)}, and in dimension one in arbitrary
characteristic \cite{HMQS}*{Proposition~4.10}.
We settle the conjecture in full generality by a different
method.

\begin{theorem}\label{thm:main}
Let $(R,\m)$ be a Noetherian local ring of dimension $d\geq1$.
If $\widehat R$ has isolated singularities, i.e., $(\widehat R)_{\fp}$ is regular for every nonmaximal
prime $\fp$, then
$\cinf(R)=1$.
Equivalently, for every $\eps>0$ there is an integer
$N=N(R,\eps)$ such that
\[
 e_R(I)\leq d!(1+\eps)\length_R(R/I)
\]
whenever $I$ is $\m$-primary and
$\length_R(R/I)>N$.
\end{theorem}

\subsection{Hypersurface sections and Bertini flags}

To prove Theorem \resultref{thm:main}, the main idea of this paper is to study the behaviour of the function $c_{\infty}$ along hypersurface sections. It turns out that one can bound $c_{\infty}(R)$ in terms of $c_{\infty}(R/xR)$ and the multiplicities on the punctured spectrum of $R$, which remains true even when the singular
locus has positive dimension. 
With  this inequality at hand, Theorem \resultref{thm:main} follows immediately from an induction argument.

  We say that a local ring is \emph{Cohen--Macaulay
on the punctured spectrum} if
$R_P$ is Cohen--Macaulay for every $P\ne\m$.  For a complete
reduced equidimensional ring $R$ of dimension
at least two, put
\[
 \muo(R):=\sup\{e(R_P):P\in\Spec R,\ P\ne\m\}.
\]
This is a finite positive integer by
\cite{Nagata}*{Theorem~40.1}; see
Lemma~\resultref{lem:punctured-multiplicity}. The following is the main technical result of this paper.

\begin{theorem}\label{thm:recurrence}
Let $(R,\m,k)$ be a complete reduced equidimensional local ring of
dimension $d\geq2$, with $k$ infinite. Assume that $R$ is
Cohen--Macaulay on the punctured spectrum.  If $x\in\m$ is
$R$-regular and
$S:=R/xR$, then
\[
 \cinf(R)\leq\max\{\cinf(S),\muo(R)\}.
\]
\end{theorem}

A priory, no reducedness or Cohen--Macaulay hypothesis is imposed on
$S$.  But we may carefully choose $x\in\m$ so that the estimate can be iterated.  For this purpose, we  introduce the following terminology.

\begin{definition}
Let $R$ be a complete reduced equidimensional
Cohen--Macaulay local ring.
An \emph{admissible Bertini flag} is a sequence
\[
 \mathcal F=(R_0,x_0;R_1,x_1;\ldots;
 R_{t-1},x_{t-1};R_t)
\]
with $R_0=R$, $R_{i+1}=R_i/x_iR_i$, such that every $R_i$ is
complete, reduced, equidimensional, and Cohen--Macaulay, every $x_i$
 is $R_i$-regular, and $R_t$ has isolated
singularities, of dimension $\ge 1$. 
\end{definition}

Notice that the empty flag is allowed when $R$ itself has
isolated singularities.  We define
\[
 \Lambda(\mathcal F):=
 \max\{1,\muo(R_0),\ldots,\muo(R_{t-1})\}.
\]

\begin{theorem}\label{thm:flag-bound}
Let $(R,\m,k)$ be a complete reduced equidimensional
Cohen--Macaulay local ring of positive dimension, with $k$
 infinite.  Then $R$ admits an admissible Bertini flag, and every
such flag satisfies
\[
 \cinf(R)\leq\Lambda(\mathcal F).
\]
In fact, a flag may be chosen with at most $s=\dim\Sing R$
sections.
\end{theorem}
The displayed inequality follows immediately from Theorem \resultref{thm:recurrence}. The existence of admissible Bertini flags follows by repeatedly applying the Flenner--Trivedi local Bertini theorem
\cites{Flenner,Trivedi,TrivediErratum}; see Lemma \resultref{lem:bertini-flag-step}.
\begin{remark}
Taking the minimum over all admissible flags, we obtain
\[
 \cinf(R)\leq\Lambda_{\mathrm{Bert}}(R)
 :=\min_{\mathcal F\ \mathrm{admissible}}
       \max\{1,\muo(R_0),\ldots,\muo(R_{t-1})\}\leq e(R).
\]
The last inequality follows by choosing each section superficial for
the maximal ideal (which uses the assumption that $k$ is infinite), so that $e(R_i)=e(R)$; see
Lemmas~\resultref{lem:punctured-multiplicity}
and~\resultref{lem:bertini-flag-step}.
\end{remark}

\begin{remark}
Ma and Smirnov \cite{UniformLech}*{Theorem~3} proved the following
uniform improvement of Lech's inequality.  If $d\geq2$ and
$e((\widehat R)_{\mathrm{red}})>1$, then there is a real number
$\delta=\delta(R)>0$ such that
\[
 e_R(I)\leq d!\bigl(e(R)-\delta\bigr)\length_R(R/I)
 \qquad(\sqrt{I}=\m),
\]
 uniformly over all
$\m$-primary ideals, but the constant $\delta(R)$ is not explicit.
In the situation of Theorem~\resultref{thm:flag-bound},
each admissible flag $\mathcal F$ determines an explicit asymptotic
coefficient:
 for every $\eps>0$, there is an integer $N=N(R,\mathcal F,\eps)$
such that
\[
 e_R(I)\leq d!\bigl(\Lambda(\mathcal F)+\eps\bigr)\length_R(R/I)
 \qquad\bigl(\sqrt{I}=\m,\ \length_R(R/I)>N\bigr).
\]
When $\Lambda(\mathcal F)<e(R)$, this improves the
classical coefficient $e(R)$ for sufficiently large colength.  Not that, however, this
condition alone does not imply
$\Lambda(\mathcal F)<e(R)-\delta(R)$.
\end{remark}

\begin{example}\label{ex:intro-balanced}
Let $k$ be an infinite field, let $a,b\geq2$, and put $m=a+b$.
If
$\operatorname{char}k=p>0$ assume further that $p\nmid mab$.
Set
$
 R=k\fps{x,y,z,w}/(x^m+y^az^b+w^m),
$
the section $y+z$ defines an admissible flag $\mathcal F$ with
\[
 \cinf(R)\leq\Lambda(\mathcal F)=\max\{a,b\}<m=e(R).
\]
In particular, taking $a=\lfloor m/2\rfloor$ and
$b=\lceil m/2\rceil$ gives
$\cinf(R)\leq\lceil m/2\rceil$.
See
Proposition~\resultref{prop:balanced-hypersurface} for the details.
\end{example}

The following example shows that the term $\muo(R)$ in Theorem~\resultref{thm:recurrence} cannot be
omitted, even for a flag of length one.

\begin{example}\label{ex:intro-ordinary}
Let $k$ be an infinite field and let
$\lambda_1,\ldots,\lambda_n\in k$ be distinct, with $n\geq3$.  Set
\[
 R=k\fps{u,v,z}\big/\Bigl(\prod_{j=1}^n(u-\lambda_jv)\Bigr).
\]
Then $z$ defines an admissible flag $\mathcal F$, and
$
 \cinf(R/zR)=1,\,
 \frac n2\leq\cinf(R)\leq\Lambda(\mathcal F)=n.
$
Thus $\cinf(R)$ is unbounded in this family, although the section
always has asymptotic Lech constant $1$.
See Proposition~\resultref{prop:ordinary-multiple-point} for the proof.
\end{example}

\subsection{Attainment of the Lech--Mumford constant}

Mumford studied the optimal constant in Lech's inequality in
connection with the singularities occurring in GIT theory.  His multiplicity--colength inequalities give
necessary local conditions for asymptotic GIT semistability
\cite{MumfordStability}*{\S3}.
As in Ma and Smirnov \cite{MaSmirnov}, for $d>0$ we define the
\emph{Lech--Mumford constant} by
\[
 \cLM(R)=\sup_{\sqrt{I}=\m}
 \frac{e_R(I)}{d!\length_R(R/I)}.
\]
 The two constants satisfy
\[
 1\leq\cinf(R)\leq\cLM(R)\leq e(R).
\]

\begin{remark}\label{rem:power-series-tail}
Let $(R,\m)$ be any Noetherian local ring, and set $d=\dim R$.  We
observe that
\[
 \cLM(R\fps{T})=\cinf(R\fps{T}).
\]
Hence $R$ is
semistable in Mumford's sense \cite{MaSmirnov}*{Definition~2.2.1} if
and only if $\cinf(R\fps{T})=1$.
Indeed, for $q\geq1$, the local map
$\varphi_q\colon A_0=R\fps{U}\to A_1=R\fps{T}$, defined by
$U\mapsto T^q$, is finite free of rank
$q$.  If $J\subset A_0$ is $(\m,U)$-primary and
$J_q=\varphi_q(J)A_1$, then
\[
 \length_{A_1}(A_1/J_q^s)
 =q\length_{A_0}(A_0/J^s)\qquad(s\geq1).
\]
This immediately implies $\length_{A_1}(A_1/J_q)
 =q\length_{A_0}(A_0/J)$ and $e(J_q)=qe(J)$.
Thus every normalized ratio occurs at arbitrarily large colength.
Combining this observation with
\cite{HMQS}*{Lemma~5.10}, we obtain
\[
 \cLM(R\fps{T})
 =\lim_{N\to\infty}
 \sup_{\substack{\sqrt I=(\m,T)\\ I\subseteq(\m,T)^N}}
 \frac{e_{R\fps{T}}(I)}{(d+1)!\length_{R\fps{T}}(R\fps{T}/I)},
\]
answering \cite{MaSmirnov}*{Question~3.3.4}. 
\end{remark}

A basic question is whether the supremum defining $\cLM(R)$ is
attained by an ideal. 
Ma and Smirnov confirmed this in a few cases \cite{MaSmirnov}*{Propositions~3.3.3 and~3.3.7} and conjectured that this holds whenever
$\widehat R$ has isolated singularities and $\cLM(R)>1$
\cite{MaSmirnov}*{Conjecture~3.3.1}.  

We can immediately deduce from
Theorems~\resultref{thm:main} and~\resultref{thm:flag-bound} the 
following result.

\begin{samepage}
\begin{theorem}\label{cor:LM-introduction}
Let $(R,\m)$ be a Noetherian local ring of dimension $d>0$.
Suppose that one of the following conditions holds:
\begin{enumerate}[label=\textup{(\roman*)}]
\item $\widehat R$ has isolated singularities and
$\cLM(R)>1$.
\item $R$ is complete, reduced, equidimensional, and
Cohen--Macaulay, with infinite residue field, and
$\cLM(R)>\Lambda(\mathcal F)$ for some admissible Bertini flag
$\mathcal F$.
\end{enumerate}
Then there is an ideal $I_0$ with $\sqrt{I_0}=\m$ such that
\[
 \cLM(R)=\frac{e_R(I_0)}{d!\length_R(R/I_0)}\in\mathbb Q.
\]
\end{theorem}
\end{samepage}

The case (i) settles the conjecture in full generality.  Notice that the condition
$\cLM(R)>1$ is essential; see
\cite{MaSmirnov}*{Remark~3.3.2(1)}.

The criterion in part~\textup{(ii)} can be verified for Veronese
subrings without first determining $\cLM$.
For a standard graded algebra $A$, write
$A^{(n)}=\bigoplus_{j\geq0}A_{nj}$, with the grading indexed by $j$.

\begin{theorem}\label{thm:veronese-attainment}
Let $k$ be an infinite perfect field, and let $A$ be a standard
graded reduced equidimensional Cohen--Macaulay $k$-algebra of
dimension $d\geq2$.  Set
\[
 R_n=\widehat{(A^{(n)})_{A^{(n)}_+}}.
\]
If $n>d!$, then $\cLM(R_n)$ is attained and rational.  If $A$
is normal and $d\geq3$, the same conclusion holds whenever
$n^2>d!$.
\end{theorem}

\begin{remark}
More generally, if $\Sing(\Proj A)\ne\varnothing$ has codimension
$c$ in $\Proj A$, the same conclusion holds whenever $n^c>d!$.
\end{remark}

The multiplicity at the vertex grows as $n^{d-1}$, whereas a
suitable Bertini flag bounds $\cinf(R_n)$ by a polynomial of
smaller degree.  This gives the strict gap required in
Theorem~\resultref{cor:LM-introduction}\textup{(ii)}.
We remark that the rings $R_n$ need not be Gorenstein or have isolated
singularities; see Example~\resultref{ex:non-gorenstein-veronese}.

\subsection{Nonisolated klt cones}

A refinement of the proof of Theorem~\resultref{thm:recurrence} (see Variant \resultref{var:LM-recurrence} and Section ~\resultref{subsec:del-pezzo-cones})
gives the following family of nonisolated klt singularities with
$\cinf=1$, in every dimension at least three.

\begin{example}[Cones over del Pezzo products]
\label{ex:del-pezzo-cones}
Let $k$ be an algebraically closed field of characteristic zero,
let $S$ be a singular del Pezzo surface with only Du Val
singularities, and let $Z$ be a smooth Fano variety of dimension
$d-3$, where $d\geq3$; for $d=3$, take $Z=\Spec k$.
Put $Y=S\times Z$, and choose $q\geq1$ so that
$L=-qK_Y$ is very ample and its section ring
\[
 A=\bigoplus_{j\geq0}H^0(Y,L^{\otimes j})
\]
is generated in degree one.  For $n\geq1$, let $R_n$ be the
completed local ring at the vertex of $\Spec A^{(n)}$.
These cones are normal and klt, and
\[
 \dim\Sing R_n=d-2,\qquad
 \cinf(R_n)=1,\qquad
 \lim_{n\to\infty}\cLM(R_n)=\infty.
\]
In fact, $\Lambda_{\mathrm{Bert}}(R_n)=2$ for every $n$, so
the equality $\cinf(R_n)=1$ improves on the Bertini-flag bound.
The proof is given in Section~\ref{subsec:del-pezzo-cones}.
\end{example}

\subsection{Outline of the proof of Theorem~\texorpdfstring{\resultref{thm:recurrence}}{\ref{thm:recurrence}}}

The proof proceeds by considering the colon
filtration of an $\m$-primary ideal $I$,
\[
 I\subseteq I:x\subseteq I:x^2\subseteq\cdots.
\]
For the stages whose images in $S=R/xR$ have large
colength, we use the one-step inequality of Huneke--Ma--Quy--Smirnov
\cite{HMQS}*{Lemma~5.1}.  Our main estimate treats the complementary
case: for every fixed $C\geq1$ and $\eps>0$,
\[
 e_R(I)\leq d!\bigl(\muo(R)+\eps\bigr)\length_R(R/I)
\]
whenever $\length_S(S/IS)\leq C$ and
$\length_R(R/I)$ is sufficiently large; see
Proposition~\resultref{prop:bounded-slice-estimate}.

The difficulty lies in intervals on which the images of
the colon ideals in $S$ are constant.  Their lengths can be
arbitrarily large, even if the common colength is bounded, so an
additive error at each step cannot be ignored.

We prove a uniform estimate for steps sufficiently far
from both ends by contradiction, using the compactness principle
underlying ultraproducts; see
Proposition~\resultref{prop:constant-block-estimate}.
We pass to a catapower $R_\#$ of $R$, the complete Noetherian local
ring obtained by quotienting an ultrapower of $R$ by its
$\m$-adic infinitesimals.  In a sequence of counterexamples, the
ultraproducts of the adjacent colon ideals $I_n$ and $I_n:x$ have
the same image in $R_\#$:
\[
 \catlim I_n=\catlim(I_n:x)=L.
\]

Moreover, $R_\#/L$ is one-dimensional Cohen--Macaulay
(Lemma~\resultref{lem:interior-block-limit}), and its generic points lie
over the punctured spectrum of $R$.

The equality of these limits means that the multiplicity
drop cannot be read directly from the catapower.
The main new ingredient is a noncancellation argument on compatible
one-dimensional sections that bounds the original drops in terms of the
limiting curve; see Section \ref{sec:branch-noncancellation}.  We can then apply
Lech's inequality in the ambient localizations at the generic points of
the limiting curve, where the multiplicities are bounded by
$\muo(R)$.  Together with the associativity formula, this contradicts
the choice of counterexamples.  Summing along the filtration proves the
displayed estimate, since the contributions near the ends are uniformly
bounded.

\subsection{Notation and conventions}

Unless explicitly stated otherwise, all rings are commutative and
Noetherian; the ultraproducts constructed below need not be Noetherian.  We write
$\length_A(-)$ for $A$-module length.  Throughout, ``primary'' means primary to
the maximal ideal of the ambient local ring.

Let $(A,\mathfrak a)$ be local, let $M$ be a nonzero finitely
generated $A$-module of dimension $t$, and let $I\subseteq\mathfrak a$
be an ideal such that $\length_A(M/IM)<\infty$.  We write
$e_A(I;M)$ for the Hilbert--Samuel multiplicity
\[
 e_A(I;M)=\lim_{q\to\infty}
 \frac{t!\,\length_A(M/I^qM)}{q^t}.
\]
When $M=A$, we write $e_A(I)=e_A(I;A)$, and we put
$e(A)=e_A(\mathfrak a)$.  We omit the ring subscript when the
ambient ring is clear.  If
$I=(a_1,\ldots,a_s)$, we also write $e_A(a_1,\ldots,a_s;M)$
for $e_A(I;M)$, omitting $M$ when $M=A$.

By the associativity formula \cite{HunekeSwanson}*{Theorem~11.2.4},
applied over $A/\operatorname{Ann}_A M$, we have
\[
 e_A(I;M)=
 \sum_{\substack{\fp\in\Supp_A M\\\dim A/\fp=t}}
 \length_{A_\fp}(M_\fp)\,e_A(I;A/\fp).
\]

For an ideal $\mathfrak a$ and an
$A$-module $M$, put
$H^0_{\mathfrak a}(M)=0:_M\mathfrak a^\infty$.  If
$P\in\Spec A$ and $n\geq1$, put $P^{(n)}=P^nA_P\cap A$, the $n$-th symbolic power of $P$.

\section{Reductions and Bertini sections}
\label{sec:reductions}

\subsection{Preliminary results}

This section collects several results which allows, say, the passage to completions and the case of infinite residue fields while without changing the multiplicity and 
colength.

\begin{lemma}
\label{lem:top-reduction}
Let $(A,\n)$ be a local ring.
\begin{enumerate}[label=\textup{(\roman*)}]
\item If $Q\subset A$ is $\n$-primary, then
\[
 \length_A(A/Q)=\length_{\widehat A}(\widehat A/Q\widehat A),
 \qquad
 e_A(Q)=e_{\widehat A}(Q\widehat A).
\]
\item Suppose that $A$ is complete and has isolated singularities,
and that $\dim A=d\geq1$.  Let $P_1,\ldots,P_s$ be the minimal primes of
dimension $d$, and set
\[
 U=\bigcap_{i=1}^sP_i,\qquad B=A/U.
\]
Then $B$ is complete, reduced, and equidimensional, has isolated
singularities, and satisfies $B/\n B\cong A/\n$.  For every
$\n$-primary ideal $Q$,
\begin{equation}
 e_A(Q)=e_B(QB),\qquad
 \length_B(B/QB)\leq\length_A(A/Q).
\label{eq:top-reduction}
\end{equation}
\end{enumerate}
\end{lemma}

\begin{proof}
Part (i) follows from faithful flatness of completion and the
isomorphisms $A/\n^r\cong\widehat A/\n^r\widehat A$.  These
isomorphisms also show that extension and contraction induce a
bijection between the $\n$-primary ideals of $A$ and the
$\n\widehat A$-primary ideals of $\widehat A$.

For (ii), by assumption $A_{P_i}$ is a
zero-dimensional regular local ring and hence a field.  Thus
$\length_{A_{P_i}}(A_{P_i})=1$.  The associativity formula gives
\[
 e_A(Q)
 =\sum_{i=1}^s e_A(Q;A/P_i)
 =e_B(QB).
\]
The length inequality follows because $B/QB$ is a quotient of
$A/Q$.  The ideal $U$ is radical,
$\Min B=\{P_i/U:1\leq i\leq s\}$, and
$\dim B/(P_i/U)=d$ for every $i$.  Thus $B$ is reduced and
equidimensional of dimension $d$; it is complete because it is a
quotient of $A$.

It remains to check that $B$ has isolated singularities.  Let $\mathfrak q$ be a
nonmaximal prime of $B$, with inverse image $P$ in $A$.
Since $P_1\cdots P_s\subseteq U\subseteq P$ and $P$ is prime,
some $P_i$ is contained in $P$.
The regular local ring $A_P$ is a domain, so
$(P_i)_P=0$, and 
therefore $U_P\subset (P_i)_P=0$.  Hence
$B_{\mathfrak q}\cong A_P$ is regular.
\end{proof}

The following lemma transfers a large-colength estimate from $B$ back
to $A$.

\begin{lemma}\label{lem:tail-transfer}
Retain the notation of Lemma~\resultref{lem:top-reduction}(ii), and fix
$\eps>0$.  Suppose that an integer $N_B$ has the property that
$
 e_B(J)\leq d!(1+\eps)\length_B(B/J)
$
for every $\n B$-primary ideal $J$ with
$\length_B(B/J)>N_B$.  Then the same inequality holds for every
$\n$-primary ideal $Q\subset A$ of colength greater than
\[
 N_A=\max\{N_B,\lceil e(B)N_B\rceil\}.
\]
\end{lemma}

\begin{proof}
Put $J=QB$.  If $\length_B(B/J)>N_B$, then
\eqref{eq:top-reduction} and the assumed bound in $B$ give
\[
 e_A(Q)=e_B(J)
 \leq d!(1+\eps)\length_B(B/J)
 \leq d!(1+\eps)\length_A(A/Q).
\]
If $\length_B(B/J)\leq N_B$, Lech's inequality
\cite{Lech} gives
\[
 e_A(Q)=e_B(J)
 \leq d!\,e(B)\length_B(B/J)
 \leq d!\,e(B)N_B
 <d!\length_A(A/Q).
\]
\end{proof}

  The following useful result allows passage to infinite residue fields without changing dimension, colength, or
multiplicity.

\begin{lemma}
\label{lem:infinite-extension}
Let $(A,\n,k)$ be a complete local ring and let $T$ be an
indeterminate.  Put
\[
 A^\natural=\widehat{A[T]_{\n A[T]}}.
\]
Then $A\to A^\natural$ is a regular, faithfully flat local
homomorphism,
the maximal ideal of $A^\natural$ is $\n A^\natural$, and its
residue field is $k(T)$.  It preserves dimension, colength, and
Hilbert--Samuel multiplicity.  If $A$ is reduced or equidimensional, then so is
$A^\natural$.  If $A$ has isolated singularities, so does
$A^\natural$.  Cohen--Macaulayness on the punctured spectrum is also
preserved.
\end{lemma}

\begin{proof}
Put $C=A[T]_{\n A[T]}$.  The map $A\to C$ is essentially smooth
and local, with closed fibre $k(T)$.  Complete local rings are
excellent \cite{Stacks}*{Tag~07QW}, and excellence is preserved by
localization and passage to a polynomial ring.  Thus $C$ is
excellent and the completion map $C\to\widehat C=A^\natural$ is
regular.  The composite $A\to A^\natural$ is therefore regular,
local, and faithfully flat.  Since the closed fibre is a field,
\[
 \length_{A^\natural}(A^\natural/QA^\natural)
   =\length_A(A/Q),\qquad
 e_{A^\natural}(QA^\natural)=e_A(Q)
\]
for every $\n$-primary $Q$; the map $A\to A^\natural$ also
preserves dimension.
Regular maps preserve reducedness.  If $A$ is equidimensional, then
\[
 \Min C=\{\mathfrak pC:\mathfrak p\in\Min A\},\qquad
 \dim C/\mathfrak pC=\dim A/\mathfrak p.
\]
Thus $C$ is equidimensional.  Since the excellent local ring $C$
is universally catenary, its completion is equidimensional by the
dimension formula.  Finally, the only prime of $A^\natural$ lying
over $\n$ is its maximal ideal.  If $\mathfrak q\ne\n A^\natural$
and $P=\mathfrak q\cap A$, then $P\ne\n$, and
$
 A_P\to(A^\natural)_{\mathfrak q}
$
is regular.  Regularity and Cohen--Macaulayness therefore ascend from
$A_P$ to $(A^\natural)_{\mathfrak q}$, proving the last assertion.
\end{proof}

\subsection{Bertini sections}

We first record a uniform multiplicity bound on the punctured
spectrum.
\begin{lemma}\label{lem:punctured-multiplicity}
Let $(A,\n)$ be a complete reduced equidimensional local ring of
positive dimension.  Then
\[
 \muo(A):=\max\{e(A_P):P\in\Spec A,\ P\ne\n\}\le e(A).
\]
\end{lemma}

\begin{proof}
The complete local ring $A$ is excellent and universally catenary.
Equidimensionality therefore gives
\[
 \dim A=\dim A_P+\dim A/P  \qquad \text{ for every } \qquad P\in\Spec A .
\]
Moreover, $P$ is analytically unramified in the terminology of
Nagata, since $A/P$ is a complete domain and hence
$\widehat{A/P}=A/P$ is reduced.  Nagata's localization inequality
\cite{Nagata}*{Theorem~40.1, p.~153} now gives
$e(A_P)\leq e(A)$, proving the inequality.
\end{proof}

We recall the Flenner--Trivedi local Bertini theorem
\cites{Flenner,Trivedi,TrivediErratum} in the form recorded in
\cite{OchiaiShimomoto}*{Theorem~1.1}.  If $(A,\n)$ is local,
$E\subseteq\n$ is an ideal, and
$\mathcal Q\subseteq\{P\in\Spec A:E\nsubseteq P\}$ is finite, there
is $x\in E$ such that
\begin{equation}
 x\notin P^{(2)}
 \quad\text{for every }P\text{ with }E\nsubseteq P,
 \qquad
 x\notin\bigcup_{Q\in\mathcal Q}Q.
\label{eq:FT}
\end{equation}
For $E=\n$, the first condition applies to every $P \neq \n$.
If $A$ is reduced and $\dim A>0$, then $x$ is $A$-regular:
indeed, $P^{(2)}=P$ for every
$P\in\Min A=\Ass A$, since $A_P$ is a field.
If, in addition, $A$ has isolated singularities and $\dim A\geq2$,
then $A/xA$ has isolated singularities: for $P$ nonmaximal and
contains $x$, the local ring $A_P$ is regular and
$x/1\in PA_P\setminus(PA_P)^2$, so $A_P/xA_P$ is regular.

Suppose now that $A$ is complete and reduced and that $k=A/\n$ is
infinite.  Fix generators $y_1,\ldots,y_s$ of $\n$ and a
coefficient field $C\subseteq A$ in equal characteristic, or a
Cohen coefficient DVR $C\subseteq A$ in mixed characteristic,
with residue field $k$.
We choose elements of the form
\[
 x_a=\sum_{i=1}^s a_i y_i,\qquad a_i\in C,
\]
with residue coefficient vector $\bar a\in k^s$.
Let $P_1,\ldots,P_h$ be the minimal primes of $A$, and put
$A_j=A/P_j$.  On a mixed-characteristic component, the map
$C\to A_j$ is still a coefficient ring map; on a
characteristic-$p$ component it factors through $C/pC=k$.
For each $A_j$, the general-element argument in
\cite{Flenner}*{Satz~4.1} in equal characteristic, or
\cite{OchiaiShimomoto}*{Theorem~4.3, pp.~139--141} in mixed
characteristic, supplies a nonempty open subset $U_j\subseteq k^s$
such that, for every lift of its points to $C^s$, the image of
$x_a$ avoids the symbolic square of every nonmaximal prime of
$A_j$.  Here we use the images of the same
generators $y_i$, which need not be minimal in $A_j$.
In the mixed-characteristic proof, their differentials still generate
the completed module of differentials, and the basic-element lemma
used there (Lemma~3.4) requires only generation.  Minimality is used
only to exclude the square of the maximal ideal on p.~140, a
condition not needed here.  The finite prime-avoidance step
(Lemma~4.2) likewise requires only generation.

The intersection $U=\bigcap_j U_j$ is nonempty and open in $k^s$.
For $\bar a\in U$, the element $x_a$ avoids $P^{(2)}$ for every
$P\ne\n$: choose $P_j\subseteq P$ and use the quotient map
$A_P\to(A_j)_{P/P_j}$, which sends $P^2A_P$ into the square of
the maximal ideal.  To impose the finite avoidance condition in
\eqref{eq:FT}, observe that
$(Q+\n^2)/\n^2$ is a proper subspace of $\n/\n^2$ for each
$Q\in\mathcal Q$, by Nakayama's lemma.
Finally, the superficial classes contain a nonempty open subset of
$\n/\n^2$, by \cite{HunekeSwanson}*{Proposition~8.5.7 and
Corollary~8.5.9}.  Pulling these conditions back along the surjection
$k^s\to\n/\n^2$ and intersecting with $U$, we obtain a nonempty
open set of simultaneous choices.  Since $k$ is infinite, it has
a $k$-rational point; any lift of its coordinates to $C$ defines
an element satisfying \eqref{eq:FT} that is also superficial for
$\n$.

Recall that $x\in\n$ is
superficial for $\n$ if, for some $c$,
$(\n^{q+1}:_A x)\cap\n^c=\n^q$ for all $q\geq c$.  When $x$ is
$A$-regular, this is equivalent to
$(\n^{q+1}:_A x)=\n^q$ for all $q\gg0$, and then
$e(A/xA)=e(A)$
\cite{HunekeSwanson}*{Proposition~11.1.9}.

\begin{lemma}\label{lem:bertini-flag-step}
Let $(A,\n)$ be a complete reduced equidimensional
Cohen--Macaulay local ring of dimension $r\geq2$.  Suppose that
$t=\dim\Sing A>0$.  There is an $A$-regular element
$x\in\n$ such that $B=A/xA$ is complete, reduced,
equidimensional, and Cohen--Macaulay of dimension $r-1$, and
\[
 \dim\Sing B\leq t-1.
\]
If $A/\n$ is infinite, $x$ may in addition be chosen superficial
for $\n$, in which case $e(B)=e(A)$.
\end{lemma}

\begin{proof}
Since $A$ is complete it is excellent, so $\Sing A $ is closed and has only finitely many irreducible
components.  Let $\mathcal Q$ be the set
of generic points of its positive-dimensional components, and choose
$x$ via local Bertini Theorem as in \eqref{eq:FT} with $E=\n$.

As explained above, $x$ is $A$-regular.  The quotient
$B$ is complete and
Cohen--Macaulay of dimension $r-1$.  In particular, it is
unmixed, hence equidimensional, and it satisfies $(S_1)$.

We next show that $B$ is generically reduced.  Let
$\overline P\in\Min B$, with inverse image
$P\subset A$.  Since $x$ is $A$-regular, $P$ has
height one and
$\dim A/P=r-1>0$; in particular, $P\ne\n$.
Suppose that $A_P$ were singular.  Choose $Q\in\mathcal Q$ with
$Q\subseteq P$.  A minimal prime of
the reduced ring $A$ is a regular point, so
$\operatorname{ht}Q\geq1$, forcing $Q=P$.  Thus $P$ is
the generic point of a positive-dimensional component of
$\Sing A$, contrary to $x\notin P$.  Hence $A_P$ is regular.
The condition $x\notin P^{(2)}$ says that
$x/1\notin(PA_P)^2$, so
$
 B_{\overline P}=A_P/xA_P
$
is a field.  Thus $B$ satisfies $(R_0)$, and Serre's criterion
$(R_0)+(S_1)$ shows that $B$ is reduced.

It remains to bound its singular locus.  If
$\overline P\ne\n B$ is singular and $P$ is its inverse image,
then $P\in\Sing A$: otherwise $A_P$ would be regular and
$x\in PA_P\setminus(PA_P)^2$ would make
$B_{\overline P}$ regular.  For each $Q\in\mathcal Q$, every
irreducible component of $V(Q)\cap V(x)$ has dimension
$\dim A/Q-1$, by catenarity applied in the domain $A/Q$.
The closed point contributes only a zero-dimensional component.
Therefore $\dim\Sing B\leq t-1$.

If $k=A/\n$ is infinite, choose $x$ superficial as
well by the preceding coefficient-space argument; then $e(B)=e(A)$.
\end{proof}

We use the following invariance of multiplicity under a flat local
homomorphism with regular closed fibre.

\begin{lemma}\label{lem:flat-local-multiplicity}
Let $(A,\mathfrak a,k_A)\to(B,\mathfrak b,k_B)$ be a flat local
homomorphism of Noetherian local rings with regular
closed fibre.  If $A$ is Cohen--Macaulay and $k_A$ is infinite,
then $B$ is Cohen--Macaulay and
\[
 e(B)=e(A).
\]
\end{lemma}

\begin{proof}
Put $u=\dim A$.  Since $k_A$ is infinite, we can choose a minimal
reduction $(a_1,\ldots,a_u)$ of $\mathfrak a$
\cite{HunekeSwanson}*{Proposition~8.3.7}; the sequence is empty
when $u=0$.  Since $A$ is Cohen--Macaulay,
\[
 \length_A(A/(a_1,\ldots,a_u))=e(A).
\]
Let $v$ be the dimension of the closed fibre, and lift a regular
system of parameters of that fibre to
$y_1,\ldots,y_v\in\mathfrak b$.  By \cite{Stacks}*{Tag~00MG},
this is a $B$-regular sequence and
$
 C=B/(y_1,\ldots,y_v)
$
is flat over $A$.  The $a_i$ therefore form a $C$-regular
sequence.  The dimension formula for flat local homomorphisms
\cite{Stacks}*{Tag~00ON} gives $\dim B=u+v$, so $B$ is
Cohen--Macaulay.

The images of the $y_i$ generate the maximal ideal of the closed
fibre.  Hence
$
 \mathfrak b=\mathfrak aB+(y_1,\ldots,y_v).
$
Since $\mathfrak a$ is integral over $(a_1,\ldots,a_u)$, the
parameter ideal
$
 Q=(a_1,\ldots,a_u,y_1,\ldots,y_v)B
$
reduces $\mathfrak b$.  It follows that
\[
 e(B)=\length_B(B/Q)
 =\length_C(C/(a_1,\ldots,a_u)C).
\]
Tensoring a composition series of $A/(a_1,\ldots,a_u)$ with the
flat local $A$-algebra $C$ gives a filtration whose nonzero
factors are all isomorphic to the simple $C$-module $k_B$.
Hence $\length_C(C/(a_1,\ldots,a_u)C)
=\length_A(A/(a_1,\ldots,a_u))=e(A)$.
\end{proof}

\section{Colon filtration, Constant blocks, and curve sections}
\label{sec:blocks}

\subsection{The colon filtration}

The following assumptions remain in force through the proof of
Theorem~\resultref{thm:recurrence}.

\noindent\textbf{Standing hypotheses.}

 Let $(R,\m,k)$ be
a \emph{complete reduced equidimensional} local ring of dimension $d\geq2$
with infinite residue field $k$.  Assume that $R$ is Cohen--Macaulay on
the punctured spectrum, and fix an $R$-regular element $x\in\m$.  Put
$
 S=R/xR,
$
which is complete of dimension $d-1$, but not necessarily reduced or Cohen--Macaulay.

For an $\m$-primary ideal $K$ and an $\m S$-primary ideal $J$,
set
\[
\begin{aligned}
 F_R(K)&=\frac{e_R(K)}{d!},&\qquad
 F_S(J)&=\frac{e_S(J)}{(d-1)!},\\
 s_x(K)&=\length_S\bigl(S/\bigl((K+(x))/(x)\bigr)\bigr).
\end{aligned}
\]
For convenience, put $F_R(R)=0$.  If $x\notin K$,
then \cite{HMQS}*{Lemma~5.1} implies that
\begin{equation}
 e_R(K)-e_R(K:x)
 \leq d\,e_S\bigl((K+(x))/(x)\bigr).
\label{eq:HMQSstep}
\end{equation}
 Dividing by $d!$
we obtain
\begin{equation}
 F_R(K)\leq F_R(K:x)+F_S((K+(x))/(x)).
 \label{eq:normalized-one-step}
\end{equation}
 If $x\in K$, we claim that this inequality still holds, with the interpretation that $K:x=R$ and $F_R(R)=0$. Indeed, we choose lifts in $K$ of a minimal
reduction $(\bar a_1,\ldots,\bar a_{d-1})$ of
$(K+(x))/(x)$ (using that the residue field is infinite).  Then
$(x,a_1,\ldots,a_{d-1})\subset K$ is a parameter ideal, so
\begin{equation}
 e_R(K)\leq e_R(x,a_1,\ldots,a_{d-1})
 =e_S(\bar a_1,\ldots,\bar a_{d-1})
 =e_S((K+(x))/(x)).
\label{eq:terminal}
\end{equation}
Here the first equality follows from the Koszul Euler characteristic,
using that $x$ is regular; the second follows because
$(\bar a_1,\ldots,\bar a_{d-1})$ reduces
$(K+(x))/(x)$.  Dividing \eqref{eq:terminal} by $d!$ gives the desired inequality
\[
 F_R(K)\leq\frac1dF_S((K+(x))/(x))\leq
 F_S((K+(x))/(x)).
\]

The next proposition combines the estimate when $s_x(K)$ is large
with a separate bound when $s_x(K)$ is bounded.  This accounts for
the maximum in Theorem~\resultref{thm:recurrence}.

\begin{proposition}\label{prop:colon-filtration-reduction}
Let $\alpha,\gamma>0$.  Suppose that there are integers
$C,M\geq1$ such that
\begin{align}
 s_x(K)>C
 &\ \Longrightarrow\
 F_S((K+(x))/(x))\leq\alpha s_x(K),\label{eq:large-slice}\\
 s_x(K)\leq C,\quad\length_R(R/K)>M
 &\ \Longrightarrow\
 F_R(K)\leq\gamma\length_R(R/K)\label{eq:bounded-slice}.
\end{align}
Then every $\Lambda>\max\{\alpha,\gamma\}$ admits an integer
$N=N(R,x,C,M,\Lambda,\alpha)$ such that
\[
 F_R(I)\leq\Lambda\length_R(R/I)
\]
for every $\m$-primary ideal $I$ with
$\length_R(R/I)>N$.
\end{proposition}

\begin{proof}
For an $\m$-primary ideal $I$, let
$q=\min\{n\in\mathbb Z_{\geq0}:x^n\in I\}$ and
 $K_j=I:x^j$, $0\le j\le q$.  Put
 $\lambda_j=\length_R(R/K_j)$, $\beta_j=s_x(K_j)$ for $j<q$, and
 $\beta_q=0$.  Then
$F_R(K_q)=\lambda_q=0$.  Multiplication by $x$ and
\eqref{eq:normalized-one-step} give
\[
 \lambda_j=\lambda_{j+1}+\beta_j,\qquad
 F_R(K_j)\leq F_R(K_{j+1})+F_S((K_j+(x))/(x)).
\]
Let
$
 r=\min\bigl(\{j<q:\beta_j\leq C\}\cup\{q\}\bigr).
$
For $D_j=F_R(K_j)-\Lambda\lambda_j$, the preceding two relations
and \eqref{eq:large-slice} give
$
 D_0\leq D_r-(\Lambda-\alpha)(\lambda_0-\lambda_r).
$
If $r=q$, the right-hand side is nonpositive.  If $r<q$ and
$\lambda_r>M$,
\eqref{eq:bounded-slice} gives $D_r\le0$, because
$\gamma<\Lambda$.  If $\lambda_r\leq M$, Lech's inequality gives
$F_R(K_r)\leq e(R)\lambda_r\leq e(R)M$.  Consequently,
\begin{align*}
 D_0
 &\leq F_R(K_r)-\Lambda\lambda_r
       -(\Lambda-\alpha)(\lambda_0-\lambda_r)\\
 &=F_R(K_r)-\alpha\lambda_r
       -(\Lambda-\alpha)\lambda_0\\
 &\leq e(R)M-(\Lambda-\alpha)\lambda_0,
\end{align*}
and the right-hand side is nonpositive whenever $\lambda_0>N$,
where
\[
 N=\left\lceil\frac{e(R)M}{\Lambda-\alpha}\right\rceil. \qedhere
\]
\end{proof}

\subsection{Constant blocks}

We analyze the intervals on which the images of the colon ideals in
$S=R/xR$ are constant.  These intervals have a simple exact form, and
their multiplicity differences can be computed on suitable
one-dimensional sections.

Fix an $\m$-primary ideal $I \subset R$, and put
\[
 q=\min\{n\in\mathbb Z_{\geq0}:x^n\in I\},\qquad
 K_j=I:x^j,\qquad J_j=(K_j+(x))/(x).
\]
Let $a\geq0$ and $h\geq1$ satisfy $a+h\leq q$. We will say that $[a,a+h)$ is a \emph{constant block} of the colon filtration if
the colengths $\length_S(S/J_j)$ are constant for
$a\leq j<a+h$.  Equivalently, since
$J_j\subseteq J_{j+1}$, the ideals $J_j$ are constant for
$a\leq j<a+h$.
Denote their common value by $J\subset S$, and put
$
 \beta=\length_S(S/J).
$
Put $A=K_a$.  After relabeling $K_{a+j}$ as $K_j$ for
$0\leq j\leq h$, we can write
\[
 A=K_0,\qquad K_j=A:x^j\quad(0\leq j\leq h),\qquad B=K_h,
\]
thanks to the transitivity identities $I: x^{a+j}=(I:x^a):x^j$.
The ideals in a constant block satisfy the following identities.

\begin{lemma}\label{lem:block}
There is an ideal $L\subset A$, generated by fixed lifts of generators
of $J$, such that
\begin{equation}
\begin{aligned}
 L+(x)&=K_j+(x) &&(0\leq j<h),\\
 K_j&=L+xK_{j+1} &&(0\leq j<h),
\end{aligned}
\label{eq:blockrec}
\end{equation}
and
\begin{equation}
 K_j=L+x^{h-j}B\quad(0\leq j\leq h).
\label{eq:blockform}
\end{equation}
Moreover
\begin{equation}
 \length_R(K_{j+1}/K_j)=\beta,\qquad
 \length_R(B/A)=h\beta.
\label{eq:block-lengths}
\end{equation}
\end{lemma}

\begin{proof}
Choose generators of $J$, lift them to $A$, and let $L$ be the
ideal generated by the lifts.  Then $L\subseteq A\subseteq K_j$ and
$L+(x)=K_j+(x)$ for $0\leq j<h$.  If $b\in K_j$, write
$b=l+xr$, with
$l\in L$.  Since $xr\in K_j$, one has
$r\in K_j:x=K_{j+1}$.  This gives the inclusion $K_j \subset L+xK_{j+1}$ and hence the equality
\eqref{eq:blockrec} as the reverse inclusion is clear.  Iteration, using $xL\subseteq L$,
proves \eqref{eq:blockform} (for $j=h$, it says $B=L+B$).

We have the following short exact sequence
\[
0\to R/K_{j+1}\overset{\cdot x}{\to}R/K_j
\to S/J\to 0.
\]
It follows that
$\beta= \length_S(S/J)= \length_R(R/K_j)-\length_R(R/K_{j+1})=\length_R(K_{j+1}/K_j)$.
Summing over $j$, we get $\length_R(B/A)=h\beta$, proving
\eqref{eq:block-lengths}.
\end{proof}

\subsection{Mixed multiplicities and curve sections}

For $\m$-primary ideals $L_1,\ldots,L_s$ and nonnegative integers
$r_1,\ldots,r_s$ with $r_1+\cdots+r_s=d$, we write
\[
 e_R\bigl(L_1^{[r_1]},\ldots,L_s^{[r_s]}\bigr)
\]
for the mixed multiplicity in which $L_j$ occurs $r_j$ times.  Thus,
for example,
\[
 e_R\bigl(I^{[d-1]},(I')^{[1]}\bigr)
 =e_R(\underbrace{I,\ldots,I}_{d-1},I').
\]

For $\m$-primary ideals $I\subseteq I'=I:x$ with $I'\ne R$, set
\[
 D_1(I,I')=
 e_R(I)-e_R\bigl(I^{[d-1]},(I')^{[1]}\bigr).
\]
Put
\[
 e_0=e_R(I),\quad
 e_1=e_R\bigl(I^{[d-1]},(I')^{[1]}\bigr),\quad e_d=e_R(I').
\]
By Teissier's mixed-multiplicity inequality
\cite{ReesSharp}, 
$e_1^d\leq e_0^{d-1}e_d$. By the monotonicity of Hilbert--Samuel
multiplicity, 
$e_d\leq e_0$.  If $t=(e_d/e_0)^{1/d}\in[0,1]$, then
$e_1\leq e_0t$, and hence
\begin{equation}
 e_R(I)-e_R(I')=e_0(1-t^d)
 \leq d e_0(1-t)\leq dD_1(I,I').
\label{eq:mixedinterp}
\end{equation}

\begin{remark}
A repeated application of \emph{loc. cit.} yields 
$e_R(I_1,\ldots,I_d)^d\leq\prod_{j=1}^d e_R(I_j)$.  Hence, for every
$N$,
\[
 \sup_{\substack{1\leq j\leq d,\ \sqrt{I_j}=\m\\
                  \length_R(R/I_j)>N}}
 \frac{e_R(I_1,\ldots,I_d)}
 {d!\bigl(\prod_{j=1}^d\length_R(R/I_j)\bigr)^{1/d}}
 =
 \sup_{\substack{\sqrt I=\m\\\length_R(R/I)>N}}
 \frac{e_R(I)}{d!\length_R(R/I)}.
\]
The reverse inequality follows by taking $I_1=\cdots=I_d$.  Thus
$\cinf(R)$ has the same tail formula in terms of mixed
multiplicities.
\end{remark}

We recall the following definition from \cite{DinhViet}*{Definition~1}, in the two-ideal case
needed here.  Let $(B,\mathfrak n)$ be
local, and let $J_0,L\subseteq B$ be $\mathfrak n$-primary ideals.
For an ideal $N\subseteq B$, write
$N:_B L^\infty=\bigcup_{q\geq1}(N:_B L^q)$.  An element $a\in L$
is \emph{a weak-$(FC)$-element} for $(J_0,L)$ if
\begin{itemize}
\item [\textup{(FC1)}] \qquad  $(a)\cap J_0^uL^{v+1}=aJ_0^uL^v$
     for all $u,v\gg 0$,
\item [\textup{(FC2)}]  \qquad $0:_B a\subseteq 0:_B L^\infty$.
\end{itemize}
It is an \emph{$(FC)$-element} if, in addition,
\begin{itemize}
 \item [\textup{(FC3)}] \qquad
 $\dim B/\bigl((a):_B L^\infty\bigr)
 =\dim B/(0:_B L^\infty)-1.$
\end{itemize}
A sequence $a_1,\ldots,a_t\in L$ is an \emph{$(FC)$-sequence} for
$(J_0,L)$ if, for every $0\leq r<t$, upon setting
\[
 B_r=B/(a_1,\ldots,a_r)B,\qquad
 J_{0,r}=J_0B_r,\qquad L_r=LB_r,
\]
where $B_0=B$, the image $\bar a_{r+1}$ of $a_{r+1}$ in $B_r$
is an $(FC)$-element for $(J_{0,r},L_r)$.  Since $L$ is
$\mathfrak n$-primary, \textup{(FC2)} is precisely filter-regularity.

To compute $D_1(I,I')$ after cutting down to dimension one, one must
choose the successive elements simultaneously so that they are
superficial for $I$, form an $(FC)$-sequence for $(I',I)$, and
reduce the image of $I$ in $S$.
The following lemma achieves these three genericity conditions simultaneously.

\begin{lemma}\label{lem:simultaneous-open}
Retain $\m$-primary ideals $I\subseteq I'=I:x$ with $I'\ne R$,
and let $b_1,\ldots,b_t$ generate $I$.  Suppose that
$a_1,\ldots,a_{i-1}$, with $1\leq i\leq d-1$, have already been
chosen and satisfy \textup{(i)}--\textup{(iii)} below with $i-1$ in
place of $i$. Let
$
 \pi\colon k^t\to V:=I/\m I
$
be the $k$-linear surjection defined by $b_1,\ldots,b_t$.  There is a Zariski dense open
subset $\Omega_i\subseteq V$, depending on the chosen initial segment,
with the following property.  Set
$U_i=\pi^{-1}(\Omega_i)\subseteq\mathbb A_k^t$.  For every
$\mathbf c_i=(c_{i\nu})\in U_i(k)$, choose arbitrary lifts
$\widetilde c_{i\nu}\in R$ and put
\[
 a_i(\mathbf c_i)=\sum_{\nu=1}^t\widetilde c_{i\nu}b_\nu.
\]
Then adjoining $a_i(\mathbf c_i)$ to the chosen initial segment yields
a sequence with all of the following properties:
\begin{enumerate}[label=\textup{(\roman*)}]
\item $a_1,\ldots,a_i$ is an $I$-superficial sequence, and
$a_j$ is filter-regular on $R/(a_1,\ldots,a_{j-1})$ for
$1\leq j\leq i$;
\item it is an initial segment of an $(FC)$-sequence for $(I',I)$
of length $d-1$, all of whose terms lie in $I$;
\item the images of $a_1,\ldots,a_i$ in $S$ are an initial segment
of a superficial sequence for $J=(I+(x))/(x)$ that can be completed
to a minimal reduction of $J$.  When $i=d-1$, they already
generate such a reduction.
\end{enumerate}
At each stage, $\Omega_i$, and hence $U_i$, may be replaced by a
nonempty principal open.
\end{lemma}

\begin{proof}
Put $A_{i-1}=(a_1,\ldots,a_{i-1})$,
\[
 B=R/A_{i-1},\qquad \mathfrak n=\m B,\qquad
 I_B=IB,\qquad I'_B=I'B.
\]
Since $A_{i-1}\subseteq I$, the quotient map induces a surjective
$k$-linear map
$
 \rho_B\colon V\to I_B/\mathfrak n I_B.
$

We first impose the conditions in \textup{(i)} and \textup{(ii)}.  By
\cite{HunekeSwanson}*{Proposition~8.5.7}, there is a dense open subset
$G_B^{\mathrm{sup}}\subseteq I_B/\mathfrak n I_B$ all of whose lifts
are superficial for $I_B$.  Applied to the multigraded data
$(I'_B,I_B)$ and the second component, the homogeneous prime-avoidance
argument in the proof of \cite{TrungVerma}*{Lemma~1.5} supplies a second
dense open subset
$
 G_B^{FC}\subseteq I_B/\mathfrak n I_B.
$
Every lift of a class in $G_B^{FC}$ is superficial in the sense of
\cite{DinhViet}*{Definition~4}, and hence is a weak-$(FC)$-element for
$(I'_B,I_B)$ by \cite{DinhViet}*{Proposition~6}.  Indeed, the bad
locus in the prime-avoidance argument is a finite union of proper linear
subspaces determined by the relevant multigraded associated primes.
Thus this condition depends only on the class modulo
$\mathfrak n I_B$, and every class in
\[
 G_B:=G_B^{\mathrm{sup}}\cap G_B^{FC}
\]
gives \textup{(i)} and the weak-$(FC)$ part of \textup{(ii)}.

We next verify \textup{(FC3)}, keeping track of the unsaturated
quotients.  Since
$I_B$ is $\mathfrak n$-primary,
$
 0:_B I_B^\infty=H^0_{\mathfrak n}(B).
$
The induction hypothesis gives
$
 \dim B/H^0_{\mathfrak n}(B)=d-i+1.
$
If $a\in I_B$ lifts a class in $G_B$, then \textup{(FC2)} implies
that $a$ is a nonzerodivisor on $B/H^0_{\mathfrak n}(B)$.
Indeed, if $ab\in H^0_{\mathfrak n}(B)$, then
$\mathfrak n^qb\subseteq 0:_Ba\subseteq H^0_{\mathfrak n}(B)$ for some $q$;
since $H^0_{\mathfrak n}(B)$ has finite length, this implies $b\in H^0_{\mathfrak n}(B)$.  Moreover,
\[
 \frac{(a):_B I_B^\infty}{aB}
 =H^0_{\mathfrak n}(B/aB),
\]
which also has finite length.
It follows that
$
 \dim B/\bigl((a):_B I_B^\infty\bigr)
 =\dim B/aB=d-i,
$ 
proving \textup{(FC3)}.  It also shows that
$\dim (B/aB)/H^0_{\mathfrak n}(B/aB)=d-i$, which is precisely the
induction statement needed for the next unsaturated quotient.

We now impose \textup{(iii)}.  Set
\[
 S_{i-1}=S/(\bar a_1,\ldots,\bar a_{i-1}),\qquad
 \mathfrak n_{i-1}=\m S_{i-1},\qquad J_{i-1}=JS_{i-1}.
\]
By the induction hypothesis in \textup{(iii)},
$\dim S_{i-1}=d-i$, and $J_{i-1}$ is
$\mathfrak n_{i-1}$-primary.  By assumption, the induced map
$
 \rho_S\colon V\to
 W_{i-1}:=J_{i-1}/\mathfrak n_{i-1}J_{i-1}
$
is surjective.  By \cite{HunekeSwanson}*{Proposition~8.5.7}, there is a dense open
subset $G_{i-1}^{S}\subseteq W_{i-1}$ whose lifts are superficial for
$J_{i-1}$.

Let $\mathcal P_{i-1}$ denote the set of minimal primes
$\mathfrak p$ of $S_{i-1}$ such that
$\dim S_{i-1}/\mathfrak p=d-i$.  For $\mathfrak p\in\mathcal
P_{i-1}$, set
\[
 L_{\mathfrak p}=
 \frac{(\mathfrak p\cap J_{i-1})+\mathfrak n_{i-1}J_{i-1}}
      {\mathfrak n_{i-1}J_{i-1}}\subseteq W_{i-1}.
\]
Each $L_{\mathfrak p}$ is a proper linear subspace.  Otherwise,
Nakayama's lemma would give $J_{i-1}\subseteq\mathfrak p$, contrary
to the fact that $J_{i-1}$ is $\mathfrak n_{i-1}$-primary and
$\mathfrak p\ne\mathfrak n_{i-1}$.  Hence
\[
 G_{i-1}^{\mathrm{red}}
 =G_{i-1}^{S}\setminus
  \bigcup_{\mathfrak p\in\mathcal P_{i-1}}L_{\mathfrak p}
\]
is dense and open in $W_{i-1}$.  Any lift
$z\in J_{i-1}$ of a class in this open satisfies
$\dim S_{i-1}/zS_{i-1}=d-i-1$; since
$G_{i-1}^{\mathrm{red}}\subseteq G_{i-1}^{S}$, it is superficial by
\cite{HunekeSwanson}*{Proposition~8.5.7(3)}.  Applying
\cite{HunekeSwanson}*{Theorem~8.6.3} to the image of $J_{i-1}$ in
$S_{i-1}/zS_{i-1}$, we obtain a superficial sequence in the quotient
that generates a minimal reduction.  Lifting this sequence to
$J_{i-1}$ completes $z$ to a superficial sequence.  By
\cite{HunekeSwanson}*{Proposition~8.6.1}, the resulting reduction lifts
first to $S_{i-1}$ and then, through the fixed superficial initial
segment, to $S$.  It has $d-1$ generators and is therefore a minimal
reduction of $J$.  Thus every class in
$G_{i-1}^{\mathrm{red}}$ preserves \textup{(iii)}.

We can now take
$
 \Omega_i=\rho_B^{-1}(G_B)\cap
            \rho_S^{-1}(G_{i-1}^{\mathrm{red}})\subseteq V,
$
which is dense and open.  Finally,
two choices of coefficient lifts for the same
$\mathbf c_i\in k^t$ produce elements of $I$ that differ by an
element of $\m I$.  They therefore determine the same class in $V$,
so all three conditions hold for arbitrary choices of lifts.  After
shrinking $\Omega_i$ to a nonempty principal open, its inverse image
$U_i$ is also principal.  Since $k$ is infinite, this open contains
a $k$-rational point.  Repeating the construction through stage
$d-1$ gives the continuations asserted in \textup{(ii)} and
\textup{(iii)}.
\end{proof}

Once cutting down to dimension one, we have the following 
 valuation formulas, which will be applied to each adjacent pair in a constant block.

\begin{lemma}\label{lem:simultaneous-curve-section}
Let $I\subseteq I'=I:x$ be $\m$-primary ideals with $I'\ne R$,
and put
$J=(I+(x))/(x)$.  One may choose
$a_1,\ldots,a_{d-1}\in I$ successively as in
Lemma~\resultref{lem:simultaneous-open} so that, for
\[
 C=R/(a_1,\ldots,a_{d-1}),\qquad
 Q=C/H^0_\m(C),
\]
$Q$ is a one-dimensional Cohen--Macaulay ring with no embedded primes,
$x$ is
$Q$-regular,
and
\begin{equation}
D_1(I,I')=e_Q(IQ)-e_Q(I'Q),\qquad
e_Q(x)=e_S(J).
\label{eq:curve}
\end{equation}

More precisely, index by $\nu$ the pairs
$(\mathfrak p_\nu,\mathfrak n_\nu)$, where
$\mathfrak p_\nu\in\Min Q$ and $\mathfrak n_\nu$ is a maximal ideal
of the finite normalization $\overline{Q/\mathfrak p_\nu}$.  Put
$
 V_\nu=(\overline{Q/\mathfrak p_\nu})_{\mathfrak n_\nu},
$
and let $v_\nu$ be the normalized valuation of the discrete valuation
ring (DVR) $V_\nu$.
Set
$
 c_\nu=\length_{Q_{\mathfrak p_\nu}}(Q_{\mathfrak p_\nu}),\,
 f_\nu=[\kappa(\mathfrak n_\nu):k].
$
For an ideal $E\subseteq R$, set
\[
 v_\nu(E)=\min\{v_\nu(\bar z):z\in E\},
\]
where $\bar z$ denotes the image of $z$ in $V_\nu$.  Then
\begin{equation}
\begin{aligned}
D_1(I,I')&=\sum_\nu c_\nu f_\nu
 \bigl(v_\nu(I)-v_\nu(I')\bigr),\\
e_Q(x)&=\sum_\nu c_\nu f_\nu v_\nu(x),\\
0&\le v_\nu(I)-v_\nu(I')\le v_\nu(x).
\end{aligned}
\label{eq:branch-valuation-formulas}
\end{equation}
In particular, $D_1(I,I')\le e_S(J)$.
\end{lemma}

\begin{proof}
Choose the tuple successively as in
Lemma~\resultref{lem:simultaneous-open}.  Write
$\mathbf a=(a_1,\ldots,a_{d-1})$.  Since $I$ is $\m$-primary,
\[
\frac{(\mathbf a):I^\infty}{(\mathbf a)}
=H^0_\m(R/(\mathbf a)).
\]
We now apply the mixed-multiplicity theorem
\cite{Viet}*{Theorem~3.4}, in the form
\cite{DinhViet}*{Theorem~3}, with
$(J_{\mathrm V},I_1)=(I',I)$ and
$(k_0,k_1)=(0,d-1)$.  The ideal $I$ is contained in no minimal
prime of the positive-dimensional reduced ring $R$, so it contains a
nonzerodivisor; in particular, $0:I^\infty=0$.  Part (i) of the cited
theorem gives
$
 e_R\bigl(I^{[d-1]},(I')^{[1]}\bigr)\ne0.
$
Part (ii), together with the displayed saturation identity, gives
$
 e_R\bigl(I^{[d-1]},(I')^{[1]}\bigr)=e_Q(I'Q).
$
Iterating superficial reduction
\cite{HunekeSwanson}*{Proposition~11.1.9} gives
$e_R(I)=e_C(IC)$.  Quotienting $C$ by the finite-length module
$H^0_\m(C)$ does not change this one-dimensional multiplicity, so
$e_R(I)=e_Q(IQ)$.

The images of $\mathbf a$ reduce $J$, so
$(x,\mathbf a)$ is $\m$-primary.  Thus $x$ is a parameter on
$C$, and
\[
 H^0_\m(C)=H^0_{(x)}(C)
 =\frac{(\mathbf a):x^\infty}{(\mathbf a)}.
\]
Consequently $x$ is regular on $Q$, and the nonzero
one-dimensional ring $Q$ is Cohen--Macaulay and has no embedded
primes.

Let $\chi$ denote Koszul Euler characteristic.  Since $\mathbf a$ is
a filter-regular sequence, $H_q(\mathbf a;R)$ has finite length for
$q>0$.  The spectral sequence of the double Koszul complex gives
\[
 \chi(\mathbf a,x;R)
 =\sum_q(-1)^q\chi(x;H_q(\mathbf a;R)).
\]
Since the Euler characteristic of multiplication by $x$ on a
finite-length module is zero, the right side reduces to $\chi(x;H_0(\mathbf a;R))=\chi(x;C)$. Similarly, additivity for the short exact sequence
\[
 0\to H^0_\m(C)\to C\to Q
 \to0
\]
 gives $\chi(x;C)=\chi(x;Q)$.  Since $x$ is
$R$-regular, the Koszul complex on $x$ resolves $S$.  Therefore
\[
\begin{aligned}
 e_Q(x)
 &=\chi(x;Q)=\chi(x;C)
 =\chi(a_1,\ldots,a_{d-1},x;R)\\
 &=\chi(x,a_1,\ldots,a_{d-1};R)\\
 &=\chi(\bar a_1,\ldots,\bar a_{d-1};S)\\
 &=e_S(\bar a_1,\ldots,\bar a_{d-1})=e_S(J).
\end{aligned}
\]

Applying the associativity formula to $Q$ and then passing to the
finite normalizations of the complete one-dimensional domains
$Q/\mathfrak p$ gives the first two formulas in
\eqref{eq:branch-valuation-formulas}; see
\cite{HunekeSwanson}*{Theorems~4.3.4, 11.2.4, and~11.2.7}.  Finally,
$xI'\subseteq I\subseteq I'$ imply
\[
  v_\nu(x)+v_\nu(I') \ge v_\nu(I)\geq v_\nu(I'),
\]
which proves the last formula and the asserted inequality.
\end{proof}

As a consequence, we obtain the following uniform bound for colon pairs with bounded
section colength.  For an integer $C\geq1$ write
$
 A_C=\frac{C^{d-1}e(S)}{(d-1)!}.
$
For $I\subseteq I'=I:x$ as above, if
$\length_S(S/J)\leq C$, then $(\m S)^C\subseteq J$, and hence
\begin{equation}
D_1(I,I')\leq e_S(J)\leq C^{d-1}e(S)=(d-1)!A_C.
\label{eq:rough-drop-bound}
\end{equation}
Together with \eqref{eq:mixedinterp}, this implies that
\[
 F_R(I)-F_R(I')\leq A_C.
\]

\section{The catapower construction}
\label{sec:catapower}

\subsection{The catapower and its localizations}

A key construction is to pass from sequences of ideals in $R$ to ideals in a single
complete Noetherian local ring, its catapower.  For this purpose, we will first review several
constructions concerning ultraproducts in commutative algebra, following \cite{Schoutens}*{Sections~2.1 and~8.1.1}..

Fix a nonprincipal ultrafilter $\mathcal U$ on $\mathbb N$.
We will say that a property holds for $\mathcal U$-almost all $n$ if the set of
indices where it holds belongs to $\mathcal U$.  For a sequence of
rings $(A_n)$, the ultraproduct $\prod_{\mathcal U}A_n$ is obtained
from $\prod_n A_n$ by identifying sequences that agree for
$\mathcal U$-almost all $n$.  We denote the class of $(z_n)$ by
$[z_n]$.  For the constant families $R$ and $k$, we consider
\[
 R^*=\prod_{\mathcal U}R,\qquad
 K=k^{\mathcal U}=\prod_{\mathcal U}k,
\]
known as the ultrapowers of $R$ and $k$.
Since each $\m^t$ is finitely generated,
$\m^tR^*=\prod_{\mathcal U}\m^t$
\cite{Schoutens}*{2.4.20}.  The ring $R^*$ is local with maximal
ideal $\m R^*$ and residue field $K$
\cite{Schoutens}*{2.4.8}. Although $R$ is complete Noetherian, its ultrapower $R^*$ need not be so \cite{Schoutens}*{2.4.16}.
We therefore pass to its separated quotient
\[
 \Inf(R^*)=\bigcap_{t\geq1}\m^tR^*,\qquad
 R_\#=R^*/\Inf(R^*),
\]
called the \emph{catapower} of $R$.  Here $\Inf(R^*)$ consists
of the classes $[z_n]$ for which $z_n\in\m^t$ for
$\mathcal U$-almost all $n$, for every fixed $t\geq1$.
By \Cref{lem:catapower-regularity} below (see also
\cite{Schoutens}*{Theorem~8.1.4}), the ring $R_\#$ is a complete Noetherian local ring, with maximal ideal $\m R_\#$
and residue field $K$.  Constant sequences induce the diagonal map
$R\to R_\#$.  For ideals $E_n\subseteq R$, let
$\catlim E_n$ denote the image of $\prod_{\mathcal U}E_n$ in
$R_\#$.

For a bounded sequence $(a_n)$ of real numbers, we write
$\lim_{\mathcal U}a_n$ for the unique real number $a$ such that
$\{n:|a_n-a|<\varepsilon\}\in\mathcal U$ for every
$\varepsilon>0$.  For any sequence $(a_n)$ of real numbers, we write
$a_n\to_{\mathcal U}\infty$ if
$\{n:a_n\geq B\}\in\mathcal U$ for every real number $B$.

The following length comparison is a special case of
\cite{Schoutens}*{Proposition~2.4.17}; we include the argument for
the quotients used here.

\begin{lemma}\label{lem:bounded-transfer}
Let $H\in\mathbb Z_{\geq0}$, and let $E_n\subseteq R$ be ideals for
which $\length_R(R/E_n)\leq H$.  Then
\begin{equation}
 \Inf(R^*)\subseteq\prod_{\mathcal U}E_n,\qquad
 R_\#/\catlim E_n\cong\prod_{\mathcal U}R/E_n,
\label{eq:bounded-transfer}
\end{equation}
and
\begin{equation}
 \length_{R_\#}(R_\#/\catlim E_n)
 =\lim_{\mathcal U}\length_R(R/E_n).
\label{eq:length-transfer}
\end{equation}
\end{lemma}

\begin{proof}
If $H=0$, then $E_n=R$ for every $n$, and all the assertions are
immediate.  Hence assume $H\geq1$.
A nonzero finite-length module of length at most $H$ is killed by
$\m^H$.  Hence $\m^H\subseteq E_n$ for every $n$, which implies
the first assertion.  Consequently,
\[
R_\#/\catlim E_n\cong \prod_{\cU}R /\prod_{\cU} E_n \cong \prod_{\mathcal U}R/E_n,
\]
see also \cite{Schoutens}*{2.1.6}.

Let $N_n=R/E_n$ and $N^*=\prod_{\mathcal U}N_n$.  For
$0\leq i<H$, finite generation of $\m^i$ gives
\[
\m^iN^*/\m^{i+1}N^*
 \cong
 \prod_{\mathcal U}
 \m^iN_n/\m^{i+1}N_n.
\]
The right-hand side is a $K$-vector space whose
dimension is the ultralimit of the corresponding $k$-dimensions.
Summing the dimensions of the successive quotients of the
$\m$-adic filtration proves
\eqref{eq:length-transfer}.
\end{proof}

By
\cite{Schoutens}*{Theorem~8.1.15 and Corollary~8.1.16}, the diagonal map $R\to R_\#$ is faithful flat and regular; the latter
applies because $R$ is complete and hence excellent.  We give a
proof via a Cohen presentation, treating equal and mixed
characteristic separately, and establish the localization properties
needed below.

\begin{lemma}\label{lem:catapower-regularity}
The ring $R_\#$ is complete Noetherian, and the diagonal
map $R\to R_\#$ is faithfully flat and regular.  If $R$ is equidimensional of dimension $d$, then $R_\#$ is reduced and
equidimensional of dimension $d$ and is catenary.  If, in addition,
$R$ is Cohen--Macaulay on the punctured spectrum, then every
nonmaximal localization
of $R_\#$ is Cohen--Macaulay.  If $R$ has isolated
singularities, these localizations are regular.
\end{lemma}

\begin{proof}
Choose a Cohen presentation (\cite{Stacks}*{Tag~032A})
$
 R=\mathcal O\fps{X_1,\ldots,X_e}/\mathfrak a,
$
where $\mathcal O=k$ in equal characteristic and $\mathcal O$ is a
Cohen ring in mixed characteristic.  In mixed characteristic with
imperfect residue field, we fix one Cohen ring; no canonicity is
required.  Since $R$ is reduced, it has characteristic zero or $p>0$. 

Let 
\[
 \mathcal O^*=\prod_{\mathcal U}\mathcal O, \quad \text{ and }  \quad 
 \mathcal O_\#=
 \begin{cases}
 K,&\mathcal O=k,\\[1mm]
 \mathcal O^*/\bigcap_{r\geq1}p^r\mathcal O^*,
   &\mathcal O\text{ is a Cohen ring}.
 \end{cases}
\]
In the second case, every nonzero class has a representative whose
$p$-adic order is constant on a set in $\mathcal U$; it is
therefore $p^r$ times a unit for a unique $r$.  In particular,
$\mathcal O_\#$ is $p$-adically separated.

To prove completeness, take a compatible family
$\eta_s\in\mathcal O_\#/p^s\mathcal O_\#$, $s\geq1$.
We can recursively choose $u_i\in\mathcal O_\#$ so that
$\sum_{i=0}^{s-1}p^iu_i$ represents $\eta_s$, and represent each
$u_i$ by a sequence $(u_{i,n})_n$ in $\mathcal O$.  Put
\[
 b_n=\sum_{i=0}^{n-1}p^iu_{i,n}.
\]
For every fixed $s$, one has
$b_n\equiv\sum_{i=0}^{s-1}p^iu_{i,n}\pmod{p^s\mathcal O}$
whenever $n\geq s$.  Since $\{n:n\geq s\}\in\mathcal U$,
the image of $[b_n]$ in $\mathcal O_\#$ maps to $\eta_s$
for every $s$.  Thus $\mathcal O_\#$ is $p$-adically complete
(cf. \cite{Schoutens}*{Theorem~8.1.4 and Remark~8.1.5}), and hence
a Cohen ring with uniformizer $p$ and residue field $K$.

The compatible reduction maps modulo $\m^N$, for $N\geq1$, induce
an isomorphism
\begin{equation}
 R_\#\overset{\sim}{\to}
 \varprojlim_N\prod_{\mathcal U}R/\m^N.
 \label{eq:inverse-limit}
\end{equation}
It is injective by definition.  For surjectivity, we use the diagonal
argument in the proof of \cite{Schoutens}*{Theorem~8.1.4}.
Take a compatible family $(\xi_N)_N$ and choose a representing
sequence at each level.  There are decreasing sets
\[
 G_1\supseteq G_2\supseteq\cdots,\qquad G_r\in\mathcal U,
\]
such that, on $G_r$, the representing sequences at levels
$1,\ldots,r$ are compatible under all relevant reduction maps.  Define
$
 r(n)=\max\bigl(\{r\leq n:n\in G_r\}\cup\{1\}\bigr).
$
For each fixed $s$, the set $G_s\cap\{n:n\geq s\}$ belongs to
$\mathcal U$, and $r(n)\geq s$ there.  Thus
$r(n)\to_{\mathcal U}\infty$.  For each $n$, choose a lift to
$R$ of the $n$-th term representing $\xi_{r(n)}$.
For each fixed $s$, these lifts agree modulo $\m^s$ with the
terms representing $\xi_s$ for $\mathcal U$-almost all $n$.
The resulting class in $R_\#$ therefore maps to $(\xi_N)_N$.
Different choices of lifts differ by elements of $\m^{r(n)}$,
so give the same class.  This proves the desired isomorphism.

Let $A=\mathcal O\fps{\mathbf X}$, let $\n$ be its maximal ideal,
and let $A'=\mathcal O_\#\fps{\mathbf X}$.  Reduction modulo
$\n^N$ involves only finitely many coefficients, so taking their
ultraproducts gives
\[
 \prod_{\mathcal U}A/\n^N
 \cong A'/\n^NA'.
\]
If $\mathfrak a=(f_1,\ldots,f_s)$, the fixed finite generating set
also gives, by \cite{Schoutens}*{2.1.6},
\[
 \prod_{\mathcal U}R/\m^N
 \cong A'/(\mathfrak aA'+\n^NA').
\]
The ideal $\mathfrak aA'$ is closed in the complete Noetherian ring
$A'$.  Taking inverse limits in \eqref{eq:inverse-limit} therefore
gives
\begin{equation}
 R_\#\cong
 \mathcal O_\#\fps{X_1,\ldots,X_e}/
 \mathfrak a\mathcal O_\#\fps{X_1,\ldots,X_e}.
\label{eq:coefficient}
\end{equation}
In particular, $R_\#$ is complete and Noetherian.  In equal
characteristic, \eqref{eq:coefficient} is the description in
\cite{Schoutens}*{Theorem~8.1.19}.

To prove regularity of $R\to R_\#$, we first show that the field
extension $K/k$ is separable.  This is automatic when
$\operatorname{char}k=0$ \cite{Stacks}*{Tag~0321}.  Suppose that
$\operatorname{char}k=p>0$.  If the $p$-monomials of a finite
$p$-independent tuple $\mathbf b$ in $k$ became dependent over
$K^p$, there would be elements $y_\nu\in K$, not all zero, such
that
\[
 \sum_\nu y_\nu^p\mathbf b^\nu=0.
\]
The existence of such a tuple is a first-order statement with
parameters in $k$.  Since $K$ is an ultrapower of $k$,
{\L}o\'s's theorem \cite{Schoutens}*{Theorem~2.3.2, pp.~14--15}
implies that such a tuple already exists in $k$.  Its $p$-th
powers would then give a relation over $k^p$, a contradiction.
Thus every finite
subset of a $p$-basis of $k/\mathbb F_p$ remains $p$-independent
in $K$.  It follows from
\cite{Stacks}*{Tag~07P2; Tag~031X}
that $K/k$ is separable.  By
\cite{Stacks}*{Tag~07EQ}, the map $k\to K$ is regular,
equivalently geometrically regular, in either characteristic.

It follows that $\mathcal O\to\mathcal O_\#$ is regular.  In equal
characteristic this is the field extension $K/k$.  In mixed
characteristic, $\mathcal O_\#$ is torsion-free over the DVR
$\mathcal O$, hence flat.  Its closed fibre is $K$, which is
geometrically regular over $k$, and its generic fibre is an extension
of characteristic-zero fields and is therefore regular.

For every $N\geq1$, extension of coefficients gives
\[
A'/\n^NA'
\cong(A/\n^N)\otimes_{\mathcal O}\mathcal O_\#.
\]
(However, it is not necessarily true that
$A'\cong A\otimes_{\mathcal O}\mathcal O_\#$.)
The modules $A'/\n^NA'$ are flat over $A/\n^N$, with surjective
transition maps.  Since $A'$ is their inverse limit, the adic flatness
criterion \cite{Stacks}*{Tag~0912} shows that $A'$ is
flat over $A$.  Its closed fibre is $K$, which is geometrically
regular over $k$.  By
\cite{Stacks}*{Tag~07PM}, the map
\[
 A=\mathcal O\fps{\mathbf X}\to
 A'=\mathcal O_\#\fps{\mathbf X}
\]
is regular.  By \eqref{eq:coefficient}, its base change along the
finite-type map $\mathcal O\fps{\mathbf X}\to R$ is
$R\to R_\#$, which is therefore regular
\cite{Stacks}*{Tag~07C1}.  Being local and flat, the latter is
faithfully flat.

The natural map $\mathcal O_\#\to R_\#$ in
\eqref{eq:coefficient} is also injective.  In equi-characteristic, $\mathcal O_\#=K$ is a field and $R_\#\ne0$, so this is automatic.  In mixed characteristic, this follows from the fact that $R$ has
characteristic zero and the map
$R\to R_\#$ is faithfully flat, hence injective. 
We henceforth identify $\mathcal O_\#$ with its image in $R_\#$.

Suppose now that $R$ is reduced and equidimensional.  Since
$R\to R_\#$ is regular, $R_\#$ is reduced
\cite{Schoutens}*{Corollary~8.1.17}.
If $Q\in\Min(\mathfrak aA')$, going down for the flat map
$A\to A'$ shows that $P=Q\cap A$ belongs to
$\Min(\mathfrak a)$.  Then
$\operatorname{ht}_A P=\dim A-d$.  Moreover, $Q$ is minimal
over $PA'$.  The dimension formula for the flat local map
$A_P\to A'_Q$ \cite{Stacks}*{Tag~00ON} gives
\[
\operatorname{ht}_{A'}Q
=\operatorname{ht}_A P+
\operatorname{ht}\bigl(Q(A'\otimes_A\kappa(P))\bigr)
=\operatorname{ht}_A P,
\]
because the prime of the fibre corresponding to $Q$ is minimal.
Since $\dim A'=\dim A$, every component of $R_\#$ has
dimension $d$.  Complete local rings are excellent, hence universally
catenary \cite{Stacks}*{Tag~07QW}.

Finally, the closed fibre $R_\#/\m R_\#=K$ is a field.  Hence the
only prime of $R_\#$ over $\m$ is its maximal ideal.  A nonmaximal
$\mathfrak q$ therefore contracts to a nonmaximal prime $P$ of
$R$, and  the localized map
$
 R_P\to(R_\#)_{\mathfrak q}
$
is regular.  Cohen--Macaulayness and regularity therefore ascend
from $R_P$ to $(R_\#)_{\mathfrak q}$
\cite{Stacks}*{Tag~045J; Tag~031E}, proving the last two assertions.
\end{proof}

\subsection{Constant blocks in the catapower}

We will choose consecutive ideals increasingly far from both ends of
constant blocks.  The following lemma shows that their images in the
catapower coincide and that the common quotient is one-dimensional
and Cohen--Macaulay.

\begin{lemma}\label{lem:interior-block-limit}
Let
\[
 K_{n,0}\subseteq K_{n,1}\subseteq\cdots\subseteq K_{n,h_n}
 \qquad(n\geq1)
\]
be constant blocks in colon filtrations, reindexed from $0$.  For
$0\leq t<h_n$, put
$
 J_n=(K_{n,t}+(x))/(x)\subseteq S;
$
this ideal is independent of $t$ because the $n$-th block is
constant.  Assume that
$
 \length_S(S/J_n)=\beta,
 \, 1\leq\beta\leq C,
$
for every $n$.  Choose indices $j_n$ with
$
 0\leq j_n<h_n,\,
 j_n\to\infty,\, h_n-j_n\to\infty.
$
Put
\[
 I_n=K_{n,j_n},\qquad I_n'=K_{n,j_n+1},\qquad
 L=\catlim I_n.
\]
Then
\begin{equation}
\begin{gathered}
\catlim I_n'=L,\qquad L:x=L,\\
T=R_\#/L\text{ is one-dimensional Cohen--Macaulay},
\qquad \length_T(T/xT)=\beta.
\end{gathered}
\label{eq:interior-block-limit}
\end{equation}
\end{lemma}

\begin{proof}
For each $n$, let $P_n$ be generated by the fixed lifts chosen
in Lemma~\resultref{lem:block}, and put
$
 B_n=K_{n,h_n},\, r_n=h_n-j_n.
$
Equation \eqref{eq:blockform} gives
\[
 I_n=P_n+x^{r_n}B_n,\qquad
 I_n'=P_n+x^{r_n-1}B_n.
\]
Since $r_n\to_{\mathcal U}\infty$, the ultraproducts of
$x^{r_n}B_n$ and $x^{r_n-1}B_n$ lie in $\Inf(R^*)$.  Hence
$\catlim P_n=\catlim I_n=\catlim I_n'=L$.

Fix $N\geq1$.  On a set in $\mathcal U$, one has
$N\leq r_n-1$, and
\[
 I_n:x^N=K_{n,j_n+N},\qquad
 I_n':x^N=K_{n,j_n+1+N}.
\]
The kernels of multiplication by $x^N$ on $R/I_n$ and $R/I_n'$
therefore have length $N\beta$ (see (\ref{eq:block-lengths})). An endomorphism of a finite-length
module has kernel and cokernel of equal length.  Hence
\begin{equation}
\length_R(R/(I_n+(x^N)))=N\beta
=\length_R(R/(I_n'+(x^N))).
\label{eq:finite-slice-lengths}
\end{equation}
The ultrapower of the fixed ideal $(x^N)$ is $x^NR^*$
\cite{Schoutens}*{2.4.20}, so
\[
 \catlim(I_n+(x^N))=L+(x^N)
 =\catlim(I_n'+(x^N)).
\]
Each successive quotient of the filtration of $R/I_n$ by powers
of $x$ is a quotient of $R/(I_n+(x))$.  Since $I_n\subseteq I_n'$,
we therefore have, for every $n$,
\[
 \length_R(R/(I_n'+(x^N)))
 \leq \length_R(R/(I_n+(x^N)))
 \leq N\length_R(R/(I_n+(x)))=N\beta.
\]
Applying Lemma~\resultref{lem:bounded-transfer} and using
\eqref{eq:finite-slice-lengths}, which holds for $\mathcal U$-almost
all $n$, we obtain
\begin{equation}
 \length_{R_\#}(T/x^NT)=N\beta\qquad(N\geq1).
\label{eq:plateau}
\end{equation}
In particular, $T/xT$ is nonzero of finite length, so
$\dim T\leq1$.

Consider the following exact sequence 
\[
 T/x^{N-1}T\overset{x}{\to}T/x^NT\to T/xT
 \to0.
\]
The source, target, and cokernel have lengths
$(N-1)\beta,N\beta,\beta$, so the first map is injective.  Thus
$(x^NT:_Tx)=x^{N-1}T$, and so $0:_Tx\subseteq x^{N-1}T$ for every $N$. By Krull's intersection theorem
\cite{Stacks}*{Tag~00IP},
$0:_Tx \subseteq \bigcap_N x^{N-1}T =0$.  Hence $L:x=L$, and the local ring $T$ is
one-dimensional and Cohen--Macaulay.
\end{proof}

In the situation of Lemma~\resultref{lem:interior-block-limit}, put
\[
\omega_x(L)=\sum_{\mathfrak q\in\Min(R_\#/L)}
e_{(R_\#)_{\mathfrak q}}(L_{\mathfrak q})\,
e_{R_\#}(x;R_\#/\mathfrak q).
\]

The associativity formula and Lech's inequality give the following
bound in terms of $\length_{R_\#}(T/xT)$.

\begin{lemma}\label{lem:limit-multiplicity-bound}
With $L,T,\beta$ as in Lemma~\resultref{lem:interior-block-limit}, one has
\begin{equation}
\omega_x(L)\le(d-1)!\muo(R)\beta,
\label{eq:limit-multiplicity-bound}
\end{equation}
where $\muo(R):=\sup\{e(R_P):P\in\Spec R,\ P\ne\m\}$.
\end{lemma}

\begin{proof}
Since $T$ is one-dimensional and Cohen--Macaulay,
$\Ass T=\Min T$.  Thus, for every
$\mathfrak q\in\Min(R_\#/L)$,
$\dim R_\#/\mathfrak q=1$ and $x\notin\mathfrak q$.
By Lemma~\resultref{lem:catapower-regularity}, $R_\#$ is
equidimensional and catenary, so
$
\operatorname{ht}\mathfrak q=d-1.
$
Since $\mathfrak q$ is nonmaximal and the closed fibre of
$R\to R_\#$ is a field, $P=\mathfrak q\cap R\ne\m$.
The domain $R/P$ has positive dimension, so its fraction field
$\kappa(P)$ is infinite.  The localized map
$
 R_P\to (R_\#)_{\mathfrak q}
$
is regular, with regular closed fibre.  Since $R_P$ is
Cohen--Macaulay, Lemma~\resultref{lem:flat-local-multiplicity} gives
\begin{equation}\label{eq:localized-catapower-multiplicity}
 \begin{gathered}
 \dim(R_\#)_{\mathfrak q}=d-1,\qquad
 (R_\#)_{\mathfrak q}\text{ is Cohen--Macaulay},\\
 e((R_\#)_{\mathfrak q})=e(R_P)\leq\muo(R).
 \end{gathered}
\end{equation}

Thus, by Lech's inequality \cite{Lech}, $e_{(R_\#)_{\mathfrak q}}(L_{\mathfrak q}) \le (d-1)!\muo(R) \length_{(R_\#)_{\mathfrak q}}(T_{\mathfrak q})$. Applying the associativity formula
\cite{HunekeSwanson}*{Theorem~11.2.4} to the one-dimensional
Cohen--Macaulay ring $T$,  we obtain
\[
\begin{aligned}
\omega_x(L)&\le(d-1)!\muo(R)\sum_{\mathfrak q}
\length_{(R_\#)_{\mathfrak q}}(T_{\mathfrak q})
 e_{R_\#}(x;R_\#/\mathfrak q)\\
&=(d-1)!\muo(R)e_{R_\#}(x;T)\\
&=(d-1)!\muo(R)\length_{R_\#}(T/xT)
 =(d-1)!\muo(R)\beta.
\end{aligned} 
\]
\end{proof}

\section{Curve sections in the catapower}
\label{sec:compatible-curves}

The goal of this section is to prove
Proposition~\resultref{prop:interior-mixed-drop}, which bounds the
mixed-multiplicity drop for steps increasingly far from both ends of a
constant block.

\subsection{Compatible choices of general elements}

We will choose elements $a_{1,n},\ldots,a_{d-1,n}$ defining the
finite-level curve sections and satisfying the genericity conditions of
Lemma~\resultref{lem:simultaneous-open} for every $n$, while their
classes in $R_\#$ generate a minimal reduction of
$L_{\mathfrak q}$ for every
$\mathfrak q\in\Min(R_\#/L)$.  The finite-level conditions define
dense open subsets over $k$, whereas the minimal-reduction conditions
are initially defined over the fields $\kappa(\mathfrak q)$.
The next two lemmas allow us to impose both sets of conditions.

\begin{lemma}\label{lem:varying-open-selection}
Let $m\in\mathbb N_{\geq1}$ be fixed, let
$V_n\subseteq\mathbb A_k^m$ be dense open
subsets, and put $K=\prod_{\cU}k$.  Given finitely many nonzero
polynomials $P_1,\ldots,P_s\in K[\mathbf c]$, there are points
$\mathbf c_n\in V_n(k)$ such that
\[
 P_i([\mathbf c_n])\ne0\qquad(1\leq i\leq s),
\]
where $[\mathbf c_n]\in K^m$ denotes the ultraproduct class.
\end{lemma}

\begin{proof}
Let $P=\prod_{i=1}^sP_i$, and choose sequence representatives for
its coefficients.  The resulting polynomials $P_n\in k[\mathbf c]$
involve a fixed finite set of monomials.  Since $P\ne0$, we have
$P_n\ne0$ for $\mathcal U$-almost all $n$.  For such $n$,
$V_n\cap D(P_n)$ is a nonempty open subset of $\mathbb A_k^m$ and
contains a $k$-rational point because $k$ is infinite.
Choose $\mathbf c_n$ in this
intersection, and choose any point of $V_n(k)$ at the remaining
indices.  Then $P([\mathbf c_n])\ne0$ by {\L}o\'s's theorem
\cite{Schoutens}*{Theorem~2.3.2}.
\end{proof}

The conditions at a minimal prime of $R_\#/L$ are initially defined
over its residue field.  We need conditions over $K$ that remain valid
for every lift to the coefficient ring.  

\begin{lemma}
\label{lem:coefficient-nonvanishing-descent}
Let $K$ be an infinite field.  Let $\mathcal O_\#=K$, or let
$\mathcal O_\#$ be a DVR with uniformizer $p$ and residue field
$K$.  In the first case put $\rho=\operatorname{id}_K$, and in the
second let $\rho\colon\mathcal O_\#\to K$ be reduction modulo $p$.
Fix $t\geq1$ and write
$\mathbf c=(c_1,\ldots,c_t)$.

Let $A$ be a ring containing $\mathcal O_\#$, let
$\Sigma\subseteq\Spec A$ be a finite subset, and let
$
 0\ne F_{\mathfrak q}\in
 \kappa(\mathfrak q)[\mathbf c]
 \,(\mathfrak q\in\Sigma).
$
Then there are nonzero
polynomials
$
 P_{\mathfrak q}\in K[\mathbf c]
 \,(\mathfrak q\in\Sigma)
$
with the following property.  If
$\overline{\mathbf c}\in K^t$ satisfies
$
 P_{\mathfrak q}(\overline{\mathbf c})\ne0
 \,(\mathfrak q\in\Sigma),
$
then every lift
$\widetilde{\mathbf c}\in\mathcal O_\#^t$ with
$\rho(\widetilde{\mathbf c})=\overline{\mathbf c}$ satisfies
\[
 F_{\mathfrak q}
 \bigl(\widetilde{\mathbf c}_{\mathfrak q}\bigr)\ne0
 \qquad(\mathfrak q\in\Sigma),
\]
where $\widetilde{\mathbf c}_{\mathfrak q}$ denotes the image of $\widetilde{\mathbf c}$ in
$\kappa(\mathfrak q)^t$.  Consequently,
\[
 \bigcap_{\mathfrak q\in\Sigma}D(P_{\mathfrak q})
 \subseteq\mathbb A_K^t
\]
is a dense open subset containing a $K$-rational point.
\end{lemma}

\begin{proof}
If $\mathcal O_\#=K$, then
$K\hookrightarrow\kappa(\mathfrak q)$.  Suppose that
$\mathcal O_\#$ is a DVR\@.  If $p\in\mathfrak q$, then
$
 \mathfrak q\cap\mathcal O_\#=(p),
 \,
 K=\mathcal O_\#/(p)\hookrightarrow\kappa(\mathfrak q).
$
If $p\notin\mathfrak q$, then
$
 \mathfrak q\cap\mathcal O_\#=(0),
 \,
 E:=\operatorname{Frac}(\mathcal O_\#)
 \hookrightarrow\kappa(\mathfrak q).
$
These embeddings also show that every
$\kappa(\mathfrak q)$ is infinite.

Fix $\mathfrak q$, and write
$
 F_{\mathfrak q}
   =\sum_\alpha a_\alpha\mathbf c^\alpha.
$
If $\mathcal O_\#=K$, put $E=K$.  If $\mathcal O_\#$ is a DVR,
put $E=K$ when $p\in\mathfrak q$, and put
$E=\operatorname{Frac}(\mathcal O_\#)$ when
$p\notin\mathfrak q$.  Let $W$ be the $E$-span in
$\kappa(\mathfrak q)$ of the coefficients $a_\alpha$; it is nonzero, finite-dimensional.  Choose
a nonzero $E$-linear functional
$\lambda\colon W\to E$. Then
$
 G_{\mathfrak q}
   :=\sum_\alpha\lambda(a_\alpha)\mathbf c^\alpha
   \in E[\mathbf c]\setminus\{0\}.
$
For every $\mathbf c\in E^t$, we have
$
 G_{\mathfrak q}(\mathbf c)
   =\lambda\bigl(F_{\mathfrak q}(\mathbf c)\bigr);
$
consequently,
\[
 G_{\mathfrak q}(\mathbf c)\neq 0
 \quad\Longrightarrow\quad
 F_{\mathfrak q}(\mathbf c)\neq 0.
\]

If $\mathcal O_\#=K$, set
$P_{\mathfrak q}=G_{\mathfrak q}$.  The same definition works when
$\mathcal O_\#$ is a DVR and $p\in\mathfrak q$: all
$\mathcal O_\#$-lifts of a fixed point of
$K^t$ have the same image in $\kappa(\mathfrak q)^t$.

It remains to treat the DVR case where
$p\notin\mathfrak q$. Write
$G_{\mathfrak q}:=\sum_\alpha b_\alpha\mathbf c^\alpha$, and put
$
 s=\min\{v_p(b_\alpha):b_\alpha\ne0\},
$
where $v_p$ denotes the
normalized $p$-adic valuation on
$\operatorname{Frac}(\mathcal O_\#)$.
Then
$
 H_{\mathfrak q}:=p^{-s}G_{\mathfrak q}
 \in\mathcal O_\#[\mathbf c]
$
has at least one unit coefficient.  Hence
$
 P_{\mathfrak q}:=
 H_{\mathfrak q}\bmod p
 \in K[\mathbf c]\setminus\{0\}.
$
For every lift $\widetilde{\mathbf c}$ of
$\overline{\mathbf c}$,
$
 H_{\mathfrak q}(\widetilde{\mathbf c})\bmod p
   =P_{\mathfrak q}(\overline{\mathbf c}).
$
If the latter is nonzero, then
$H_{\mathfrak q}(\widetilde{\mathbf c})$ is a unit.  Thus
$G_{\mathfrak q}(\widetilde{\mathbf c})\ne0$, and the preceding implication gives
$F_{\mathfrak q}
(\widetilde{\mathbf c}_{\mathfrak q})\ne0$.

Finally, the product of the finitely many
$P_{\mathfrak q}$ is nonzero.  Since $K$ is infinite, the
nonvanishing locus of this product contains a $K$-rational point.
\end{proof}

We now impose the finite-level conditions of
Lemma~\resultref{lem:simultaneous-curve-section} together with the
minimal-reduction conditions at the minimal primes of $R_\#/L$.

\begin{lemma}\label{lem:compatible-coefficient-choice}
In the situation of Lemma~\resultref{lem:interior-block-limit}, assume
in addition that
$I_n'\ne R$ for every $n$.  We choose the tuples
$
 \mathbf a_n=(a_{1,n},\ldots,a_{d-1,n})\in I_n^{d-1}
$
 as in Lemma~\resultref{lem:simultaneous-curve-section}.
For $1\leq i\leq d-1$, let
$a_i\in R_\#$ denote the class of the sequence $(a_{i,n})_n$, and
set $\mathbf a=(a_1,\ldots,a_{d-1})$.  Then:
\begin{enumerate}[label=\textup{(\roman*)}]
\item for every $n$, the tuple $\mathbf a_n$ satisfies
\eqref{eq:curve}--\eqref{eq:branch-valuation-formulas};
\item the images of $\mathbf a_n$ modulo $x$ generate a minimal
reduction of $J_n$;
\item for every $\mathfrak q\in\Min(R_\#/L)$, the ideal
$(\mathbf a)(R_\#)_{\mathfrak q}$ is a minimal reduction of
$L_{\mathfrak q}$.
\end{enumerate}
\end{lemma}

\begin{proof}
Recall the constant blocks
$K_{n,0}\subseteq\cdots\subseteq K_{n,h_n}$ and the indices $j_n$
of Lemma~\resultref{lem:interior-block-limit}.  Thus
$I_n=K_{n,j_n}$, $I_n'=K_{n,j_n+1}$, and $L=\catlim I_n$.
Put $B_n=K_{n,h_n}$ and $r_n=h_n-j_n$.  For each $n$, let
$P_n\subseteq K_{n,0}$ be generated by lifts to $K_{n,0}$ of a
minimal generating set of
$J_n=(K_{n,0}+(x))/(x)\subseteq S=R/xR$.
Lemma~\resultref{lem:block} then gives
\begin{equation}
 I_n=P_n+x^{r_n}B_n,\qquad r_n\to_{\mathcal U}\infty,
\label{eq:block-tail-form}
\end{equation}
Recall also that $\length_S(S/J_n)=\beta\leq C$.  Hence
$(\m S)^C\subseteq J_n$, and
\[
 \mu_S(J_n):=\dim_k(J_n/(\m S)J_n)
 \leq\length_S(S/(\m S)^{C+1}).
\]
Let $t$ be a common bound for $\mu_S(J_n)$, and pad the lifted
generating sets with zeros to obtain $\ell_{1,n},\ldots,\ell_{t,n}$.
Let $g_{1,n},\ldots,g_{u_n,n}$ generate $x^{r_n}B_n$; the number
$u_n$ may vary with $n$.  For $1\leq\nu\leq t$, let
$\ell_\nu\in R_\#$ denote the class of $(\ell_{\nu,n})_n$.
Since $t$ is fixed, taking ultraproducts of the generating maps
\cite{Schoutens}*{2.1.6} gives
\begin{equation}
 L=\catlim P_n=(\ell_1,\ldots,\ell_t)R_\#.
\label{eq:limit-generators}
\end{equation}
Let $\mathcal O$ be the coefficient field or Cohen ring obtained in the
proof of Lemma~\resultref{lem:catapower-regularity}; it is a subring of $R_\#$.

We will construct the sequences $(a_{i,n})_n$ by induction on
$i=1,\ldots,d-1$.  Suppose that, for every $n$, there exist
$a_{1,n},\ldots,a_{i-1,n}$ satisfying
\textup{(i)}--\textup{(iii)} of
Lemma~\resultref{lem:simultaneous-open}, with $i-1$ in place of $i$.
Let $a_1,\ldots,a_{i-1}\in R_\#$ be their classes, and assume also
that their images in $(R_\#)_{\mathfrak q}$ can be completed to a
minimal reduction of $L_{\mathfrak q}$ for every
$\mathfrak q\in\Min(R_\#/L)$.  Both requirements are vacuous when
$i=1$.

To construct the next elements, we apply
Lemma~\resultref{lem:simultaneous-open} at each $n$ to the chosen
initial segment and the generating set
$
 \ell_{1,n},\ldots,\ell_{t,n},g_{1,n},\ldots,g_{u_n,n}
$
of $I_n$.  Thus we obtain a nonempty principal open $
 D(H_{i,n})\subseteq\mathbb A_k^{t+u_n}
$ in the resulting set of
admissible residue-field coefficient vectors.
Let $(\mathbf c_i,\mathbf d_i) \in \bA_k^t \times_k \bA_k^{u_n}=\mathbb A_k^{t+u_n}$, expand
\[
 H_{i,n}(\mathbf c_i,\mathbf d_i)
   =\sum_{\gamma}h_{\gamma,i,n}(\mathbf c_i)\mathbf d_i^\gamma
\]
and set
\[
 V_{i,n}=\bigcup_\gamma D(h_{\gamma,i,n})
 \subseteq\mathbb A_k^t.
\]
This is a dense open subset of $\mathbb A_k^t$.  For every
$\mathbf c_i\in V_{i,n}(k)$, the polynomial
$H_{i,n}(\mathbf c_i,-)$ is nonzero, so there
exists $\mathbf d_i$ such that
$(\mathbf c_i,\mathbf d_i)\in D(H_{i,n})$ (because $k$ is infinite).

We next impose the condition at each
$\mathfrak q\in\Min(R_\#/L)$.
By \eqref{eq:localized-catapower-multiplicity}, the local ring
$(R_\#)_{\mathfrak q}$ is Cohen--Macaulay of dimension $d-1$,
and $L_{\mathfrak q}$ is $\mathfrak q(R_\#)_{\mathfrak q}$-primary.
To express the minimal-reduction condition in terms of residue-field
coefficients, we use the fibre cone.  Let $\tau$ be an indeterminate, and consider
the Rees algebra $(R_\#)_{\mathfrak q}[L_{\mathfrak q}\tau]$, with degree-$r$ part $L_{\mathfrak q}^r\tau^r$.
 The fibre cone, by definition, is its quotient by the ideal generated by
$\mathfrak q(R_\#)_{\mathfrak q}$:
\[
 \mathcal F(L_{\mathfrak q})
 =\bigoplus_{r\geq0}
L_{\mathfrak q}^r/
 \mathfrak q(R_\#)_{\mathfrak q}L_{\mathfrak q}^r.
\]
 The
degree-one part is
$\mathcal F(L_{\mathfrak q})_1
=L_{\mathfrak q}/\mathfrak q(R_\#)_{\mathfrak q}L_{\mathfrak q}$.
Since $L_{\mathfrak q}$ is primary in a $(d-1)$-dimensional local
ring, a tuple of $d-1$ elements of $L_{\mathfrak q}$ generates a
minimal reduction if and only if its images in
$\mathcal F(L_{\mathfrak q})_1$ form a homogeneous system of
parameters of $\mathcal F(L_{\mathfrak q})$
\cite{HunekeSwanson}*{Propositions~8.2.4 and~8.3.7}.

For $1\leq j<i$, let $u_j$ denote the image of $a_j/1$ in $\mathcal F(L_{\mathfrak q})_1$. Put
$
 \mathcal G_{i,\mathfrak q}
 =\mathcal F(L_{\mathfrak q})/(u_1,\ldots,u_{i-1});
$
this quotient inherits the grading, and we write
$(\mathcal G_{i,\mathfrak q})_1$ for its degree-one part.
By the induction hypothesis, $u_1,\ldots,u_{i-1}$ can be completed
to a homogeneous system of parameters consisting of linear forms; hence
$\dim\mathcal G_{i,\mathfrak q}=d-i\geq1$.  Let
$\mathfrak P_1,\ldots,\mathfrak P_s$ be the top-dimensional
homogeneous minimal primes of $\mathcal G_{i,\mathfrak q}$.  By
\eqref{eq:limit-generators}, the initial forms of
$\ell_1,\ldots,\ell_t$ give a surjection
$
 \kappa(\mathfrak q)^t\to
 (\mathcal G_{i,\mathfrak q})_1.
$
For every $a$, the inverse image of
$(\mathfrak P_a)_1:=\mathfrak P_a\cap(\mathcal G_{i,\mathfrak q})_1$
is a proper linear subspace: otherwise
$\mathfrak P_a$ would contain the irrelevant ideal of the standard
graded ring $\mathcal G_{i,\mathfrak q}$, contradicting
$\dim(\mathcal G_{i,\mathfrak q}/\mathfrak P_a)=d-i\geq1$.
For each $a$, choose a nonzero linear form $f_a$ in $\mathbf c_i$ vanishing on the inverse image of $(\mathfrak P_a)_1$, and set $F_{i,\mathfrak q}:= \prod_{a=1}^s f_a \in\kappa(\mathfrak q)[\mathbf c_i]$.  On the principal
open
$D(F_{i,\mathfrak q})$
the next initial form avoids every $\mathfrak P_a$.  The enlarged
initial segment can therefore be completed to a homogeneous system
of parameters consisting of linear forms, and hence to a minimal
reduction of $L_{\mathfrak q}$.

Applying  Lemma~\resultref{lem:coefficient-nonvanishing-descent}, with
$A=R_\#$ and
$\Sigma=\Min(R_\#/L)$, to the polynomials
$F_{i,\mathfrak q}$, we obtain nonzero polynomials
$
 P_{i,\mathfrak q}\in K[\mathbf c_i]
 \,
 \bigl(\mathfrak q\in\Min(R_\#/L)\bigr)
$
whose simultaneous nonvanishing at a point of $K^t$ guarantees the
required minimal-reduction condition for every $\mathcal O_\#$-lift
of that point.

Applying Lemma~\resultref{lem:varying-open-selection} to the open subsets
$V_{i,n}$ and the finitely many polynomials $P_{i,\mathfrak q}$, we obtain coefficients
$\mathbf c_{i,n}\in V_{i,n}(k)$ whose class
$[\mathbf c_{i,n}]\in K^t$ satisfies all of these
nonvanishing conditions.  Choose coefficients
$\mathbf d_{i,n}$ so that
$(\mathbf c_{i,n},\mathbf d_{i,n})\in D(H_{i,n})$.
Lift the entries of $\mathbf c_{i,n}$ to $\mathcal O$ and those of
$\mathbf d_{i,n}$ to $R$, retaining the same notation.
Define
\begin{equation}
a_{i,n}=\sum_{\nu=1}^t c_{i\nu,n}\ell_{\nu,n}
       +\sum_{\mu=1}^{u_n}d_{i\mu,n}g_{\mu,n},
\label{eq:full-parameters}
\end{equation}
and let $a_i\in R_\#$ be the class of $(a_{i,n})_n$.
Note that the second sum satisfies
\[
 \sum_\mu d_{i\mu,n}g_{\mu,n}
 \in x^{r_n}B_n\subset\m^{r_n},
\]
which has zero image in $R_\#$.  Thus
$a_i=\sum_{\nu=1}^t c_{i\nu}\ell_\nu$, where
$\mathbf c_i=(c_{i1},\ldots,c_{it})\in\mathcal O_\#^t$ is the
class of the lifted coefficient vectors.  Its residue class is
$[\mathbf c_{i,n}]\in K^t$, so
Lemma~\resultref{lem:coefficient-nonvanishing-descent} ensures that
$a_1,\ldots,a_i$ can be completed to a minimal reduction of
$L_{\mathfrak q}$ for every $\mathfrak q\in\Min(R_\#/L)$,
regardless of the chosen lifts.  At each finite level, membership in
$D(H_{i,n})$ ensures that adjoining $a_{i,n}$ preserves the
conditions of Lemma~\resultref{lem:simultaneous-open}.  This completes
the inductive step.

After $i=d-1$, Lemma~\resultref{lem:simultaneous-curve-section}
gives \textup{(i)}, condition \textup{(iii)} of
Lemma~\resultref{lem:simultaneous-open} gives \textup{(ii)}, and
the local induction gives \textup{(iii)}.
\end{proof}

\subsection{The cataproduct of the curve quotients}

In general, saturation by $x$ (i.e., the functor $I\mapsto I:x^{\infty}$) does not commute with passage to the catapower,
because the finite-level saturation exponents need not be bounded.  We
therefore take the cataproduct of the saturated curve quotients, that
is, the quotient of their ultraproduct by its ideal of infinitesimals
\cite{Schoutens}*{Section~8.1.1}. 

\begin{lemma}
\label{lem:curve-cataproduct}
For the tuples chosen in Lemma~\resultref{lem:compatible-coefficient-choice}, put
$
 M_n=R/(\mathbf a_n),\,
 \mathfrak h_n=(\mathbf a_n):x^\infty,$ and
 $Q_n=R/\mathfrak h_n.$
Then $x$ is $Q_n$-regular and, with
$\mathcal M_C:=C^{d-1}e(S)$,
\begin{equation}
 \mu_n:=\length_R(Q_n/xQ_n)=e_S(J_n)\leq\mathcal M_C.
\label{eq:bounded-curve}
\end{equation}
Put
\[
 Q^*=\prod_{\mathcal U}Q_n,\qquad
 \Inf(Q^*)=\bigcap_{r\ge1}\m^rQ^*,\qquad
 Q_\#=Q^*/\Inf(Q^*).
\]
Then the following assertions hold.
\begin{enumerate}[label=\textup{(\roman*)}]
\item the natural map $R_\#\to Q_\#$ is surjective, so
\begin{equation}
 Q_\#=R_\#/H,\qquad
 H:=\ker(R_\#\to Q_\#);
\label{eq:curve-cataproduct-kernel}
\end{equation}
\item $x$ is $Q_\#$-regular,
\begin{equation}
 Q_\#/xQ_\#\cong\prod_{\mathcal U}(Q_n/xQ_n),
 \qquad
 \length_{R_\#}(Q_\#/xQ_\#)
 =\lim_{\mathcal U}\mu_n\leq\mathcal M_C,
\label{eq:slice-transfer}
\end{equation}
and $Q_\#$ is a one-dimensional Cohen--Macaulay with no
embedded primes;
\item the tuple $\mathbf a$ annihilates $Q_\#$, and hence
\begin{equation}
 R_\#/(\mathbf a)\twoheadrightarrow Q_\#;
\label{eq:curve-quotient-surjection}
\end{equation}
\item for every $\mathfrak q\in\Min(R_\#/L)$,
\begin{equation}
 \length_{(R_\#)_{\mathfrak q}}((Q_\#)_{\mathfrak q})
 \leq e_{(R_\#)_{\mathfrak q}}(L_{\mathfrak q}),
\label{eq:localcoeff}
\end{equation}
where the left-hand side is zero if the localization vanishes.
\end{enumerate}
\end{lemma}

\begin{proof}
Since $(x,\mathbf a_n)$ is $\m$-primary, the argument in
Lemma~\resultref{lem:simultaneous-curve-section} gives
\[
 H^0_\m(R/(\mathbf a_n))
 =H^0_{(x)}(R/(\mathbf a_n))
 =\mathfrak h_n/(\mathbf a_n).
\]
Thus $Q_n$ is exactly the curve quotient used in
Lemma~\resultref{lem:simultaneous-curve-section}, and
\eqref{eq:bounded-curve} follows.  Since
$Q_n/xQ_n$ has length at most
$\mathcal M_C$,
\begin{equation}
 \m^{\mathcal M_C}Q_n\subseteq xQ_n,\qquad
 \m^{r\mathcal M_C}Q_n\subseteq x^rQ_n\subseteq\m^rQ_n
 \quad(r\geq1).
\label{eq:uniform-topologies}
\end{equation}
The quotient maps induce a surjection
$R^*\twoheadrightarrow Q^*$ \cite{Schoutens}*{2.1.6}.
Finite generation of $\m^s$ also gives
\[
 \m^sQ^*=\prod_{\mathcal U}\m^sQ_n
 \qquad(s\geq1).
\]
Thus $\Inf(R^*)$ maps into $\Inf(Q^*)$, and the induced map
$R_\#\to Q_\#$ is surjective.  This proves
\eqref{eq:curve-cataproduct-kernel}.

Suppose $xz=0$ in $Q_\#$, and choose representatives $z_n\in Q_n$.
For any fixed $r$, on a set in $\mathcal U$ that may depend on
$r$,
$
 xz_n\in\m^{(r+1)\mathcal M_C}Q_n
 \subset x^{r+1}Q_n.
$
Since $x$ is $Q_n$-regular, we have
$z_n\in x^rQ_n\subset\m^rQ_n$.  This holds for every $r$, so
$[z_n]\in\Inf(Q^*)$, and hence $z=0$.
Thus $x$ is $Q_\#$-regular.

The quotient maps induce a surjection
$Q^*\to\prod_{\mathcal U}(Q_n/xQ_n)$ with kernel $xQ^*$
\cite{Schoutens}*{2.1.6}.  Since
$\Inf(Q^*)\subseteq xQ^*$ by \eqref{eq:uniform-topologies}, the
induced map from $Q_\#$ has kernel $xQ_\#$.  This proves the
isomorphism in \eqref{eq:slice-transfer}. The equality of lengths follows from
Lemma~\resultref{lem:bounded-transfer}, applied to the ideals
$(\mathfrak h_n,x)$.
Being a quotient of $R_\#$, the ring $Q_\#$ is Noetherian.
Since $Q_\#/xQ_\#$ has finite positive length,
$\dim Q_\#=\operatorname{depth}Q_\#=1$, and $ Q_\#$ is Cohen--Macaulay with no embedded primes.

Every $a_{i,n}$ annihilates $Q_n$, so the class $a_i\in R_\#$
lies in $H$.  This proves \eqref{eq:curve-quotient-surjection}.  Fix
$\mathfrak q\in\Min(R_\#/L)$ and write
$A=(R_\#)_{\mathfrak q}$.  By
Lemma~\resultref{lem:compatible-coefficient-choice}, $(\mathbf a)A$ is a minimal reduction of
$L_{\mathfrak q}$, and
\eqref{eq:localized-catapower-multiplicity} shows that $A$ is
Cohen--Macaulay of dimension $d-1$.  The ideal $L_{\mathfrak q}$ is
$\mathfrak qA$-primary, so $(\mathbf a)A$ is a parameter ideal.
Since $A$ is Cohen--Macaulay,
\[
\length_A((Q_\#)_{\mathfrak q})
\leq\length_A\bigl(A/(\mathbf a)A\bigr)
=e_A\bigl((\mathbf a)A\bigr)
 =e_A(L_{\mathfrak q}),
\]
which is \eqref{eq:localcoeff}.
\end{proof}

\begin{remark}
There is an inclusion
$(\mathbf a):x^\infty\subset H$, but equality need not hold because
the finite-level saturation exponents may be unbounded.
\end{remark}

\subsection{Noncancellation along normalization branches}
\label{sec:branch-noncancellation}

For each $n$, let $\alpha$ range over pairs
$(\mathfrak p_{\alpha,n},\mathfrak n_{\alpha,n})$, where
$\mathfrak p_{\alpha,n}\in\Min Q_n$ and
$\mathfrak n_{\alpha,n}$ is a maximal ideal of the
normalization of $Q_n/\mathfrak p_{\alpha,n}$.  Let
$v_{\alpha,n}$ be the normalized valuation of the corresponding
discrete valuation ring.  We also view this as a valuation on $R$ through
$R\to Q_n/\mathfrak p_{\alpha,n}$, with the convention
$v_{\alpha,n}(0)=+\infty$. Put
\[
 c_{\alpha,n}
 =\length_{(Q_n)_{\mathfrak p_{\alpha,n}}}
   ((Q_n)_{\mathfrak p_{\alpha,n}}),\qquad
 f_{\alpha,n}=[\kappa(\mathfrak n_{\alpha,n}):k].
\]
For such a branch, write
\[
 w_{\alpha,n}=c_{\alpha,n}f_{\alpha,n},\qquad
 u_{\alpha,n}=v_{\alpha,n}(x),\qquad
 \kappa_{\alpha,n}=w_{\alpha,n}u_{\alpha,n}.
\]
 Note that by \eqref{eq:branch-valuation-formulas},
\begin{equation}
 \sum_\alpha\kappa_{\alpha,n}
 =e_{Q_n}(x)=\length_R(Q_n/xQ_n)\leq\mathcal M_C:=C^{d-1}e(S).
\label{eq:branch-x-contributions}
\end{equation}
Thus there are at most $\mathcal M_C$ branches, and all the integers
$w_{\alpha,n}$, $u_{\alpha,n}$, and
$\kappa_{\alpha,n}$ are uniformly bounded.  

For an ideal $E\subseteq R$,
write $v_{\alpha,n}(E):=\min \{v_{\alpha,n}(e): e\in E\}$, as in Lemma~\resultref{lem:simultaneous-curve-section}.
For the pair $I_n\subset I_n'$, put
$
 \delta_{\alpha,n}
 =v_{\alpha,n}(I_n)-v_{\alpha,n}(I_n'),
$
which satisfies $0\leq\delta_{\alpha,n}\leq u_{\alpha,n}$ by (\ref{eq:branch-valuation-formulas}).

We will say that the branch $\alpha$ is \emph{active} if $\delta_{\alpha,n}>0$.
After padding the sums with zero terms and restricting to a
$\mathcal U$-large set, we may use a fixed finite index set and
assume that the integers
$w_{\alpha,n},u_{\alpha,n},\kappa_{\alpha,n},\delta_{\alpha,n}$
and the active index set are independent of $n$.

For each $f\in R_\#$, choose a representative $(f_n)_n\in \prod_{n}R$.
For an active branch $\alpha$, we define
\begin{equation}
\mathfrak q_\alpha=
\left\{f\in R_\#:
\frac{v_{\alpha,n}(f_n)}{v_{\alpha,n}(x)}
\to_{\mathcal U}\infty\right\}.
\label{eq:asymptotic-branch-prime}
\end{equation}
Notice that the defining condition is independent of the representative.  Indeed,
\eqref{eq:branch-x-contributions} gives
$u_{\alpha,n}\leq\mathcal M_C$, whereas $v_{\alpha,n}$ is centered
at the maximal ideal and therefore $v_{\alpha,n}(\m)\geq1$.  Hence
\begin{equation}
 \frac{v_{\alpha,n}(\m)}{v_{\alpha,n}(x)}
 \geq\frac1{\mathcal M_C}.
\label{eq:valuation-topology}
\end{equation}
The difference of two representatives of the same element of $R_\#$
lies in $\m^r$ for $\mathcal U$-almost all $n$, for each fixed $r$.
Its normalized valuation therefore tends to infinity by
\eqref{eq:valuation-topology}. Thus, by the ultrametric inequality, \eqref{eq:asymptotic-branch-prime} holds for one representative if
and only if it holds for the other.

By the ultrametric inequality again, $\mathfrak q_\alpha$ is closed under addition, and since the branch valuations take nonnegative values on $R$, it is closed under
multiplication by elements of $R_\#$.  Thus $\mathfrak q_\alpha$
is an ideal, and it is proper since $x\notin\mathfrak q_\alpha$.
To see that it is prime, suppose that $f,g\notin\mathfrak q_\alpha$.
Since $\mathcal U$ is an ultrafilter, their normalized valuations are
bounded on respective sets in $\mathcal U$, hence both are bounded
on the intersection.  The normalized valuation of $fg$, being their
sum, is bounded there as well.  Hence $fg\notin\mathfrak q_\alpha$,
proving primality.

The following lemma bounds the total contribution of
the active branches associated to a fixed minimal prime of $Q_\#$.
Its proof uses the bounded curve sections, but not the constant-block
decomposition.
\begin{lemma}\label{lem:consecutive-exponent-noncancellation}
Let the $Q_n$'s be as in Lemma~\resultref{lem:curve-cataproduct}, with
$\length_R(Q_n/xQ_n)\leq M$ for an integer
$M\geq1$.  For $\mathfrak q\in\Spec R_\#$, let
$\mathcal A_{\mathfrak q}$ be the set of active branches $\alpha$
for which $\mathfrak q_\alpha=\mathfrak q$, and put
\[
 W_{\mathfrak q}:=
 \lim_{\mathcal U}\sum_{\alpha\in\mathcal A_{\mathfrak q}}
 w_{\alpha,n}\delta_{\alpha,n}.
\]
If $\mathfrak q\in\Min Q_\#$, then
\begin{equation}
W_{\mathfrak q}\leq
\length_{(R_\#)_{\mathfrak q}}((Q_\#)_{\mathfrak q})
e_{R_\#}(x;R_\#/\mathfrak q).
\label{eq:limit-prime-bound}
\end{equation}
\end{lemma}

\begin{proof}
The set of minimal primes of $Q_\#$ is finite.  If $\mathfrak q$ is
the unique minimal prime, set $f=0$.  
Otherwise, choose
by prime avoidance
\[
f\in\mathfrak q\setminus
\bigcup_{\mathfrak p\in\Min Q_\#,\ \mathfrak p\ne\mathfrak q}\mathfrak p.
\]
Let $(f_n)_n$ be a representative for $f$.  After modifying it outside
a $\mathcal U$-large set, we may assume that $f_n\in\m$ for every $n$.

On $R_\#/\mathfrak q$, the perturbation $f+x^r$
has multiplicity $r e_{R_\#}(x;R_\#/\mathfrak q)$, whereas its
multiplicities on the other components are eventually constant as
$r\to\infty$.  At finite levels, choosing among consecutive
exponents prevents cancellation on every branch.  For each fixed
interval of exponents, the resulting parameter quotients have length
bounded independently of $n$, allowing their lengths to be
transferred to $Q_\#$.

Fix $N\geq1$ and set
$
\mathcal E_N:=\{N,N+1,\ldots,N+M\}.
$
For each $n$ and each branch valuation $v$ of $Q_n$,
$v(f_n+x^r)>\min\{v(f_n),r\,v(x)\}$ is possible only if
$v(f_n)=r\,v(x)$.  Since $v(x)>0$, each branch excludes at most
one $r\in\mathcal E_N$.  There are at most $M$ branches, so we may
choose $r_{N,n}\in\mathcal E_N$ such that
\begin{equation}
v(f_n+x^{r_{N,n}})
=\min\{v(f_n),r_{N,n}v(x)\} \quad \text{ on every branch. } \quad 
 \label{eq:min}
 \end{equation}
 Put
$
 \widetilde g_{N,n}:=f_n+x^{r_{N,n}}
$
for every $n$, and set
$t_N:=\lim_{\mathcal U}(r_{N,n}-N)\in\{0,\ldots,M\}$.
Then $(\widetilde g_{N,n})_n$ represents
$
 g_N=f+x^{N+t_N}\in R_\#.
$
By \eqref{eq:min}, $\widetilde g_{N,n}$ has finite valuation on every
branch and hence avoids every minimal prime of $Q_n$.
Since $\widetilde g_{N,n}\in\m$ and $Q_n$ is
one-dimensional Cohen--Macaulay, it is a
$Q_n$-regular parameter, and
$Q_n/\widetilde g_{N,n}Q_n$ has finite length.
For every $n$, the one-dimensional associativity formula and
\eqref{eq:min} give
\begin{equation}
\begin{aligned}
\length_R(Q_n/\widetilde g_{N,n}Q_n)
&=\sum_\alpha w_{\alpha,n}
  \min\{v_{\alpha,n}(f_n),
        r_{N,n}u_{\alpha,n}\}\\
&\leq (N+M)\sum_\alpha
  w_{\alpha,n}u_{\alpha,n}
\leq(N+M)M.
\end{aligned}
\label{eq:bounded-g}
\end{equation}
For $H_N:=(N+M)M$, we then have
$\m^{H_N}Q_n\subseteq\widetilde g_{N,n}Q_n$ for every $n$.
Since $\m^{H_N}$ is finitely generated, we have
\[
 \Inf(Q^*)\subseteq\m^{H_N}Q^*
 =\prod_{\mathcal U}\m^{H_N}Q_n
 \subseteq\prod_{\mathcal U}\widetilde g_{N,n}Q_n.
\]
The quotient maps $Q_n\to Q_n/\widetilde g_{N,n}Q_n$ induce a
surjection from $Q^*$ with kernel
$\prod_{\mathcal U}\widetilde g_{N,n}Q_n$
\cite{Schoutens}*{2.1.6}.  This kernel contains $\Inf(Q^*)$, and
its image in $Q_\#$ is $g_NQ_\#$.  Hence
\begin{equation}
 Q_\#/g_NQ_\#
 \cong\prod_{\mathcal U}Q_n/\widetilde g_{N,n}Q_n.
\label{eq:g-transfer}
\end{equation}
Lemma~\resultref{lem:bounded-transfer}, applied to the ideals
$(\mathfrak h_n,\widetilde g_{N,n})\subset R$, gives
\begin{equation}
 \length_{R_\#}(Q_\#/g_NQ_\#)
 =\lim_{\mathcal U}
 \length_R(Q_n/\widetilde g_{N,n}Q_n)<\infty.
\label{eq:g-length-transfer}
\end{equation}

Since $Q_\#$ is one-dimensional Cohen--Macaulay, $g_N$ is a
$Q_\#$-regular parameter, and
$
 \length_{R_\#}(Q_\#/g_NQ_\#)=e_{R_\#}(g_N;Q_\#),
$
so the one-dimensional associativity formula applies.

For every active branch $\alpha\in\mathcal A_{\mathfrak q}$, we have
$v_{\alpha,n}(f_n)/u_{\alpha,n}\to_{\mathcal U}\infty$ by the definition of $\mathfrak q_\alpha$.  Therefore,
for this fixed $N$, the equality
$
 \min\{v_{\alpha,n}(f_n),(N+t_N)u_{\alpha,n}\}
 =(N+t_N)u_{\alpha,n}
$
holds simultaneously for all $\alpha\in\mathcal A_{\mathfrak q}$ on
a set in $\mathcal U$, because this set of branches is finite.  Their
total contribution therefore satisfies
\begin{equation}
\lim_{\mathcal U}\length_R(Q_n/\widetilde g_{N,n}Q_n)
\geq (N+t_N)
\lim_{\mathcal U}\sum_{\alpha\in\mathcal A_{\mathfrak q}}
 w_{\alpha,n}\delta_{\alpha,n}
=(N+t_N)W_{\mathfrak q}.
\label{eq:active-branch-contribution}
\end{equation}
Here we used $\delta_{\alpha,n}\leq u_{\alpha,n}$.
Combining \eqref{eq:g-length-transfer} with the associativity formula
for $Q_\#$,
\eqref{eq:active-branch-contribution} yields
\begin{equation}
(N+t_N)W_{\mathfrak q}
\leq\sum_{\mathfrak p\in\Min Q_\#}
\length_{(R_\#)_{\mathfrak p}}((Q_\#)_{\mathfrak p})
e_{R_\#}(g_N;R_\#/\mathfrak p).
\label{eq:special-assoc}
\end{equation}
Since $f\in\mathfrak q$, the
$\mathfrak q$-term on the right-hand side of
\eqref{eq:special-assoc} is
$
(N+t_N)
\length_{(R_\#)_{\mathfrak q}}((Q_\#)_{\mathfrak q})
e_{R_\#}(x;R_\#/\mathfrak q).
$
For $\mathfrak p\ne\mathfrak q$, the image of $f$ in the
one-dimensional domain $R_\#/\mathfrak p$ is nonzero.  For every branch
valuation $v$ arising from its finite normalization,
$rv(x)>v(f)$ for all sufficiently large $r$.  Hence
$
 e_{R_\#}(f+x^r;R_\#/\mathfrak p)
 =e_{R_\#}(f;R_\#/\mathfrak p)
$
for $r\gg0$.  These other terms in
\eqref{eq:special-assoc} are therefore bounded independently of $N$.
Dividing by $N+t_N$ and letting $N\to\infty$, we
obtain \eqref{eq:limit-prime-bound}.
\end{proof}

For the interior steps considered below, every active
branch defines a prime in
$\Min(R_\#/L)\cap\Min Q_\#$.
Lemma~\resultref{lem:consecutive-exponent-noncancellation}, together
with \eqref{eq:localcoeff}, then bounds
$\lim_{\mathcal U}D_1(I_n,I_n')$ by $\omega_x(L)$, which is
estimated in Lemma~\resultref{lem:limit-multiplicity-bound}.

\begin{proposition}\label{prop:interior-mixed-drop}
In the situation of Lemma~\resultref{lem:interior-block-limit},
assume in addition that $h_n-j_n\geq2$ for every $n$.  Then
\begin{equation}
\limsup_{n\to\infty}D_1(I_n,I_n')
\leq(d-1)!\muo(R)\beta.
\label{eq:interior-mixed-drop}
\end{equation}
\end{proposition}

\begin{proof}
The quantities $D_1(I_n,I_n')$ are bounded by
\eqref{eq:rough-drop-bound}.  We pass to a subsequence along which they
converge to their limsup, and reindex the subsequence by $\mathbb N$.
Fix a nonprincipal ultrafilter $\mathcal U$ on $\mathbb N$.  The
hypothesis $h_n-j_n\geq2$ implies that
$I_n'=K_{n,j_n+1}\ne R$ for every $n$.  Put
$
 L:=\catlim I_n.
$
Choose $\mathbf a_n$ as in
Lemma~\resultref{lem:compatible-coefficient-choice}, and form
$Q_n$, $Q_\#$, and
$H=\ker(R_\#\to Q_\#)$ as in
Lemma~\resultref{lem:curve-cataproduct}.  By Lemma~\resultref{lem:simultaneous-curve-section},
\[
 D_1(I_n,I_n')
 =\sum_\alpha w_{\alpha,n}\delta_{\alpha,n}.
\]
As above, pad the branch sums with zero terms and restrict to a
$\mathcal U$-large set on which the index set, the active set, and the integers
$w_{\alpha,n},u_{\alpha,n},\kappa_{\alpha,n},\delta_{\alpha,n}$ are
constant.

Let $\alpha$ be an active branch.  With the notation
\eqref{eq:block-tail-form}, put $r_n=h_n-j_n$; then
$
 I_n=P_n+x^{r_n}B_n,\,
 I_n'=P_n+x^{r_n-1}B_n,\,
 r_n\to_{\mathcal U}\infty.
$
Write
$
 p_n=v_{\alpha,n}(P_n),\,
 b_n=v_{\alpha,n}(B_n),\,
 u_n=v_{\alpha,n}(x).
$
If $p_n\leq(r_n-1)u_n+b_n$, then $v_{\alpha,n}(I_n)=v_{\alpha,n}(I_n')=p_n$, contrary to the assumption that $\alpha$ is active.
Hence
\[
 p_n>(r_n-1)u_n+b_n,\qquad
 \frac{p_n}{u_n}\geq r_n-1
\]
on a set in $\mathcal U$.  Since $r_n\to_{\mathcal U}\infty$, we
have $p_n/u_n\to_{\mathcal U}\infty$.  Since
$L=\catlim P_n$, the definition of $\mathfrak q_\alpha$ therefore implies that
$L\subseteq\mathfrak q_\alpha$.  Since
$x\notin\mathfrak q_\alpha$ and $R_\#/L$ is one-dimensional local,
$\mathfrak q_\alpha/L \subset R_\#/L$ is not a maximal ideal and hence minimal.  

We also have $H\subseteq\mathfrak q_\alpha$.  Indeed, if $(f_n)_n$ represents $f\in H$, then for each fixed $s\geq1$ the
image of $f_n$ in $Q_n$ belongs to $\m^sQ_n$ for
$\mathcal U$-almost all $n$.
It follows from \eqref{eq:valuation-topology} that
\[
 \frac{v_{\alpha,n}(f_n)}{v_{\alpha,n}(x)}
 \geq\frac{s}{\mathcal M_C}
\]
on such a set.  Since $s$ is arbitrary, $f\in\mathfrak q_\alpha$ and so $H\subseteq\mathfrak q_\alpha$.
It follows that $\mathfrak q_\alpha\in\Supp Q_\#$.  The ring
$Q_\#$ is one-dimensional and Cohen--Macaulay, and
$\dim R_\#/\mathfrak q_\alpha=1$.  Thus
$\mathfrak q_\alpha\in\Min Q_\#$, and
Lemma~\resultref{lem:consecutive-exponent-noncancellation} applies.

For each $\mathfrak q\in\Min(R_\#/L)$, we set
\[
 W_{\mathfrak q}
 :=\lim_{\mathcal U}
 \sum_{\substack{\alpha\text{ active}\\
                   \mathfrak q_\alpha=\mathfrak q}}
 w_{\alpha,n}\delta_{\alpha,n}.
\]
Since the indexes of the branches are finite, fixed, we have
\[
\begin{aligned}
\lim_{\mathcal U}D_1(I_n,I_n')
&=\sum_{\mathfrak q\in\Min(R_\#/L)}W_{\mathfrak q}\\
&\leq\sum_{\mathfrak q\in\Min(R_\#/L)}
 \length_{(R_\#)_{\mathfrak q}}((Q_\#)_{\mathfrak q})
 e_{R_\#}(x;R_\#/\mathfrak q)\\
&\leq\sum_{\mathfrak q\in\Min(R_\#/L)}
 e_{(R_\#)_{\mathfrak q}}(L_{\mathfrak q})
 e_{R_\#}(x;R_\#/\mathfrak q)\\
&=\omega_x(L)\leq(d-1)!\muo(R)\beta.
\end{aligned}
\]
The first inequality follows from
Lemma~\resultref{lem:consecutive-exponent-noncancellation}, since every
prime with a positive branch contribution lies in $\Min Q_\#$; the second
from \eqref{eq:localcoeff}, and the last from
Lemma~\resultref{lem:limit-multiplicity-bound}.  By the choice of
subsequence, this ultralimit equals the original limsup, proving the proposition.
\end{proof}

\section{Uniform estimates for constant blocks}
\label{sec:endpoint}

Using Proposition~\resultref{prop:interior-mixed-drop}, we obtain
uniform estimates for constant blocks and prove
Theorem~\resultref{thm:recurrence}.

Recall the notation $F_R(I)=e_R(I)/d!$,
$F_S(J)=e_S(J)/(d-1)!$, $F_R(R)=0$, and
$A_C=C^{d-1}e(S)/(d-1)!$ from Section~\ref{sec:blocks}.

\begin{proposition}\label{prop:constant-block-estimate}
For every $C\in\mathbb N_{\geq1}$ and every $\delta>0$, there is an
integer $H=H(R,x,C,\delta)\geq1$ with the following property.  If
$[a,a+h)$ is a constant block with common colength
$1\leq\beta\leq C$, then every step indexed by
$
 a+H\leq j\leq a+h-H-1
$
satisfies
\begin{equation}
D_1(K_j,K_{j+1})\le(d-1)!(\muo(R)+\delta)\beta.
\label{eq:interior}
\end{equation}
Every such constant block $[a,a+h)$, including blocks of length less
than $2H$ and blocks containing the terminal step,
satisfies
\begin{equation}
F_R(K_a)-F_R(K_{a+h})
\le(\muo(R)+\delta)h\beta+2HA_C.
\label{eq:block-sum-bound}
\end{equation}
\end{proposition}

\begin{proof}
Suppose that no $H\geq1$ satisfies \eqref{eq:interior}.  For each
$n$, we could then choose an $\m$-primary ideal $I^{(n)}$, its colon
filtration $K_t^{(n)}=I^{(n)}:x^t$, a constant block
$[a_n,a_n+h_n)$ with common colength
$1\leq\beta_n\leq C$, and an index $i_n$ at which
\eqref{eq:interior} fails, with
\[
 i_n-a_n\geq n,
 \qquad a_n+h_n-1-i_n\geq n.
\]
Reindex the block by
\[
 \widetilde K_{n,t}:=K_{a_n+t}^{(n)}\quad(0\leq t\leq h_n),
 \qquad j_n:=i_n-a_n.
\]
Thus $[0,h_n)$ is a constant block for the reindexed filtration,
$j_n\to\infty$, and $h_n-j_n\to\infty$; after deleting finitely
many terms, $h_n-j_n\geq2$.  We also pass to a subsequence on which
$\beta_n=\beta$ is constant.  Since the integers
$D_1(\widetilde K_{n,j_n},\widetilde K_{n,j_n+1})$ violate
\eqref{eq:interior}, each is at least
$
 \bigl\lfloor(d-1)!(\muo(R)+\delta)\beta\bigr\rfloor+1.
$
But this contradicts Proposition~\resultref{prop:interior-mixed-drop} that
\[
\limsup_n D_1(\widetilde K_{n,j_n},\widetilde K_{n,j_n+1})
\le(d-1)!\muo(R)\beta.
\]
  Thus \eqref{eq:interior} holds for some $H\geq1$.

For the block $[a,a+h)$, write $J=(K_a+(x))/(x)$, and let $q$
be the terminal index of the colon filtration.  Put
$
 \Delta_j=d!\bigl(F_R(K_j)-F_R(K_{j+1})\bigr).
$
Thus $\Delta_j=e_R(K_j)-e_R(K_{j+1})$ at a nonterminal step, whereas
$\Delta_{q-1}=e_R(K_{q-1})$ at the terminal step
$K_{q-1}\subset K_q=R$.
At an interior step, \eqref{eq:mixedinterp} and
\eqref{eq:interior} give
\[
\frac{\Delta_j}{d!}
\le\frac{D_1(K_j,K_{j+1})}{(d-1)!}
\le(\muo(R)+\delta)\beta.
\]
For the remaining nonterminal steps, it folows from \eqref{eq:mixedinterp} and
\eqref{eq:rough-drop-bound} that $\Delta_j/d!\le A_C$.
If the block contains the terminal step, then \eqref{eq:terminal} implies
\[
\frac{e_R(K_{q-1})}{d!}
\le\frac{e_S(J)}{d!}\le A_C.
\]
At most $2H$ steps lie outside the interior range, even when $h<2H$.
Summing these bounds proves
\eqref{eq:block-sum-bound}.
\end{proof}

Summing \eqref{eq:block-sum-bound} over the constant blocks gives the
following estimate.

\begin{proposition}\label{prop:bounded-slice-estimate}
For an $\m$-primary ideal $I \subset R$, put
\[
\begin{gathered}
q=\min\{n\in\mathbb Z_{\geq0}:x^n\in I\},\qquad
K_j=I:x^j,\qquad J_j=(K_j+(x))/(x),\\
\beta_j=\length_S(S/J_j)\quad(0\leq j<q),\qquad
\beta_q=0.
\end{gathered}
\]
For every $\eta>0$ and $C\in\mathbb N_{\geq1}$, let
$H=H(R,x,C,\eta/2)$ be as in \Cref{prop:constant-block-estimate}.
Then every $\m$-primary ideal $I$ satisfying
$\beta_0\leq C$ and $\length_R(R/I)>\frac{4CHA_C}{\eta}$
also satisfies
\[
 F_R(I)\leq(\muo(R)+\eta)\length_R(R/I).
\]
\end{proposition}

\begin{proof}
Note that if $\beta_j=\beta_{j+1}$ then $J_j=J_{j+1}$.
Thus the maximal intervals of constant colength are precisely the constant blocks.  There are at most $C$ such blocks, since
$\beta_0\leq C$ and $(\beta_j)_{j<q}$ is a nonincreasing sequence
of positive integers.  By the additivity of length for the colon exact sequences, we have
\[
\length_R(R/I)
=\sum_{j=0}^{q-1}\length_S(S/J_j)
=\sum_{\text{blocks}}h\beta.
\]
Telescoping \eqref{eq:block-sum-bound} over the blocks gives
\[
F_R(I)\le(\muo(R)+\eta/2)\length_R(R/I)+2CHA_C.
\]
By assumption,
$2CHA_C<(\eta/2)\length_R(R/I)$, proving the assertion.
\end{proof}

We can now finish the proof of \Cref{thm:recurrence}.
\begin{proof}[Proof of Theorem~\resultref{thm:recurrence}]
Put
$
 \Gamma=\max\{\cinf(S),\muo(R)\},
$
and fix $\Lambda>\Gamma$.  Choose
$
 \cinf(S)<\alpha<\Lambda,\,
 \muo(R)<\gamma<\Lambda.
$
By the definition of $\cinf(S)$, there is an integer $C\geq1$
such that
$
 F_S(J)\leq\alpha\length_S(S/J)
$
whenever $J$ is $\m S$-primary and
$\length_S(S/J)>C$.  Define $\eta=\gamma-\muo(R)$, and set
\[
 H:=H(R,x,C,\eta/2),\qquad
 M:=\left\lceil\frac{4CHA_C}{\eta}\right\rceil.
\]
Since $s_x(K)=\length_S(S/((K+(x))/(x)))$, we deduce from
Proposition~\resultref{prop:bounded-slice-estimate} the following implication
\[
 s_x(K)\leq C,\quad\length_R(R/K)>M
 \quad\Longrightarrow\quad
 F_R(K)\leq\gamma\length_R(R/K).
\]
The two hypotheses of
Proposition~\resultref{prop:colon-filtration-reduction} are therefore
satisfied with the constants $\alpha$ and $\gamma$.  Thus
$
 F_R(I)\leq\Lambda\length_R(R/I)
$
for every $\m$-primary ideal $I$ of sufficiently large colength.
Since $\Lambda>\Gamma$ was arbitrary,
\[
 \cinf(R)\leq\max\{\cinf(S),\muo(R)\}.\qedhere
\]
\end{proof}

\begin{variant}[of Theorem~\resultref{thm:recurrence}]
\label{var:LM-recurrence}
Let $R,x,S$ be as in Theorem~\resultref{thm:recurrence}, and let
$b\geq1$.  Suppose that $\cLM(B)\leq b$ for every regular local
homomorphism $R_P\to B$, where $x\notin P$ and $B$ is an
excellent local ring of dimension $d-1$.  Then
\[
 \cinf(R)\leq\max\{\cinf(S),b\}.
\]
\end{variant}

\begin{proof}
Use the notation $L,\beta,\omega_x(L)$ of
Lemma~\resultref{lem:limit-multiplicity-bound}.  For
$\mathfrak q\in\Min(R_\#/L)$, put $B=(R_\#)_{\mathfrak q}$ and
$P=\mathfrak q\cap R$.  Then $B$ is excellent
of dimension $d-1$, $R_P\to B$ is regular, and $x\notin P$.
The hypothesis therefore gives
\[
 e_B(L_{\mathfrak q})
 \leq(d-1)!b\length_B(B/L_{\mathfrak q}).
\]
By the associativity formula, we have
$\omega_x(L)\leq(d-1)!b\beta$.
The proofs of Propositions~\resultref{prop:interior-mixed-drop},
\resultref{prop:constant-block-estimate}, and
\resultref{prop:bounded-slice-estimate} use $\muo(R)$ only through
this bound.  Replacing $\muo(R)$ by $b$ in those proofs and in the proof of
Theorem~\resultref{thm:recurrence} proves the assertion.
\end{proof}

\section{Applications}
\label{sec:consequences}

We first prove the asymptotic Lech conjecture in arbitrary
characteristic using Theorem~\resultref{thm:recurrence}.

\begin{proof}[Proof of Theorem~\resultref{thm:main}]
We argue by induction on the dimension.  The base case $d=1$ follows
from \cite{HMQS}*{Proposition~4.10}.

Assume $d\geq2$ and that the theorem is known in dimension $d-1$.
Let $(A,\n)$ be a $d$-dimensional local ring whose completion has
isolated singularities.   Put
\[
 B:=\widehat A/
 \bigcap_{\substack{P\in\Min\widehat A\\ \dim\widehat A/P=d}}P,
\qquad
 R:=B^\natural,
\]
and denote the maximal ideal of $R$ by $\m$.  By
Lemmas~\resultref{lem:top-reduction} and~\resultref{lem:infinite-extension},
$R$ is complete, reduced, and equidimensional of dimension $d$, has
isolated singularities, and has infinite residue field.

If the residue characteristic is $p>0$ and $p$ is nilpotent in
$A$, then $pB=0$, since $B$ is reduced.  In this case $R$
has characteristic $p$.

Take $E=\m$ and $\mathcal Q=\varnothing$ in \eqref{eq:FT}, and set
$
 S=R/xR.
$
As explained after \eqref{eq:FT}, $x$ is $R$-regular and $S$ is
complete of dimension $d-1$ with isolated singularities.
By the induction hypothesis, $\cinf(S)=1$.  Since $R$ has
isolated singularities, $\muo(R)=1$, and
Theorem~\resultref{thm:recurrence} gives
\[
 \cinf(R)\leq\max\{\cinf(S),\muo(R)\}=1.
\]
Thus $\cinf(R)=1$ by \eqref{eq:universal-tail-bounds}.
Fix $\eps>0$, and choose a threshold $N_R$ for the corresponding
estimate over $R$.

For every $\n B$-primary ideal $J$, extension to $R=B^\natural$
preserves its colength and multiplicity.  Hence the same threshold is
valid for $B$.  Lemma~\resultref{lem:tail-transfer} gives a threshold
for $\widehat A$, and completion transfers the estimate to $A$.
This completes the induction.
\end{proof}

\begin{proof}[Proof of Theorem~\resultref{thm:flag-bound}]
If $R$ has isolated singularities, the empty flag is
admissible.  Otherwise, put $s=\dim\Sing R>0$.  We apply
Lemma~\resultref{lem:bertini-flag-step} successively until the quotient
has isolated singularities.  At each step the ring remains complete,
reduced, equidimensional, and Cohen--Macaulay, while the dimension of
its singular locus decreases by at least one.  Hence at most $s$
steps are needed.  The terminal ring has positive
dimension (since a reduced one-dimensional local ring has
isolated singularities), so the sequence is an admissible Bertini flag.

Fix any admissible flag $\mathcal F$.  By
Theorem~\resultref{thm:main}, its terminal ring satisfies
$\cinf(R_t)=1$.  For $0\leq i<t$, the ring $R_i$ has dimension
at least two and is Cohen--Macaulay, so
Theorem~\resultref{thm:recurrence} applies to
$R_i\to R_{i+1}$.  Thus
\[
 \cinf(R_i)\leq
 \max\{\muo(R_i),\cinf(R_{i+1})\}.
\]
Iteration gives
\[
 \cinf(R)\leq
 \max\{1,\muo(R_0),\ldots,\muo(R_{t-1})\}
 =\Lambda(\mathcal F). \qedhere
\]
\end{proof}

Since admissible flags exist and every $\Lambda(\mathcal F)$ is a
positive integer, the minimum
\[
 \Lambda_{\mathrm{Bert}}(R)
 :=\min_{\mathcal F\ \mathrm{admissible}}\Lambda(\mathcal F)
\]
is well-defined.  By Lemma~\resultref{lem:bertini-flag-step}, the flag constructed above
may be chosen so that $e(R_i)=e(R)$ for every $i$.
By Lemma~\resultref{lem:punctured-multiplicity}, 
$\Lambda(\mathcal F)\leq e(R)$.  Thus
\[
 \cinf(R)\leq\Lambda_{\mathrm{Bert}}(R)\leq e(R).
\]
The strict inequality $\cLM(R)>\cinf(R)$ ensures attainment.

\begin{proposition}\label{prop:attainment-above-tail}
Let $(A,\n)$ be a positive-dimensional Noetherian local ring.  If
$\cLM(A)>\cinf(A)$,
then $\cLM(A)$ is attained at an $\n$-primary ideal.  In particular,
$\cLM(A)\in\mathbb Q$.
\end{proposition}

\begin{proof}
Put $d=\dim A$, and choose $\cinf(A)<\lambda<\cLM(A)$.
By the definition of $\cinf(A)$, there is an integer $N\geq1$
such that every $\n$-primary ideal $I$ of colength greater than
$N$ satisfies $e_A(I)/(d!\length_A(A/I))\leq\lambda$.
For $\n$-primary ideals $I$
of colength at most $N$, Lech's inequality gives
\[
 1\leq\length_A(A/I)\leq N,\qquad
 1\leq e_A(I)\leq d!e(A)N,
\]
so only finitely many ratios occur in this range.  Since
$\lambda<\cLM(A)$, one of them exceeds $\lambda$, and their
maximum must equal $\cLM(A)$.  This proves both assertions.
\end{proof}

\begin{corollary}\label{cor:LM-attainment}
Let $(R,\m)$ be a positive-dimensional Noetherian local ring such
that $\widehat R$ has isolated singularities.  If
$\cLM(R)>1$, then $\cLM(R)$ is attained and rational.  Thus
\cite{MaSmirnov}*{Conjecture~3.3.1} holds in arbitrary characteristic
when $\widehat R$ has isolated singularities.
\end{corollary}

\begin{proof}
By Theorem~\resultref{thm:main}, $\cinf(R)=1$.  Apply
Proposition~\resultref{prop:attainment-above-tail}.
\end{proof}

\begin{corollary}\label{cor:flag-attainment}
Under the hypotheses of Theorem~\resultref{thm:flag-bound}, fix an
admissible flag $\mathcal F$.  If
$\cLM(R)>\Lambda(\mathcal F)$, then $\cLM(R)$ is attained and
rational.
\end{corollary}

\begin{proof}
By Theorem~\resultref{thm:flag-bound},
$\cinf(R)\leq\Lambda(\mathcal F)$.  It remains to apply
Proposition~\resultref{prop:attainment-above-tail}.
\end{proof}

The following criterion uses powers of the maximal ideal.

\begin{proposition}\label{prop:numerical-attainment}
Let $(R,\m)$ be a complete reduced equidimensional Cohen--Macaulay
local ring of dimension $d\geq2$, with infinite residue field, and
suppose that $\dim\Sing R\leq1$.  If
\[
 s^d e(R)>d!\muo(R)\length_R(R/\m^s)
\]
for some integer $s\geq1$, then $\cLM(R)$ is attained and
rational.  In particular, it is enough that $e(R)>d!\muo(R)$.
\end{proposition}

\begin{proof}
Theorem~\resultref{thm:flag-bound} gives a flag of length at most one,
so $\cinf(R)\leq\muo(R)$.  On the other hand,
\[
 \cLM(R)\geq
 \frac{e_R(\m^s)}{d!\length_R(R/\m^s)}
 =\frac{s^d e(R)}{d!\length_R(R/\m^s)}
 >\muo(R).
\]
It remains to apply Proposition~\resultref{prop:attainment-above-tail}.
\end{proof}

In particular, this applies to every complete three-dimensional
normal Cohen--Macaulay local ring with infinite residue field for
which $e(R)>6\muo(R)$.

\begin{proof}[Proof of Theorem~\resultref{thm:veronese-attainment}]
For a standard graded $k$-algebra $C$, write $e(C)=e(C_{C_+})$.
We first show that if $C$ is reduced, equidimensional, and
Cohen--Macaulay of dimension at least two, then
\begin{equation}\label{eq:veronese-punctured-bound}
 \muo\bigl(\widehat{(C^{(n)})_{C^{(n)}_+}}\bigr)\leq e(C)
 \qquad(n\geq1).
\end{equation}
Veronese subrings preserve reducedness, equidimensionality, and
Cohen--Macaulayness.  For the last assertion, the decomposition of
$C$ by degrees modulo $n$ exhibits $C^{(n)}$ as a direct summand
of the finite $C^{(n)}$-module $C$; apply local cohomology at the
homogeneous maximal ideal.  These properties also hold after
localization at that ideal and completion.

Put $B=C^{(n)}$, $S=\widehat{B_{B_+}}$, and
$T=\widehat{C_{C_+}}$.  Let $Q$ be a nonmaximal prime of $S$,
and put $\mathfrak q=Q\cap B$.  Then $\mathfrak q\ne B_+$, since
$B_+S=\m_S$.  The maps from the punctured cones of $B$ and $C$
to $X=\Proj B=\Proj C$ are locally given by Laurent polynomial
extensions.  Let $x\in X$ be the image of
$\mathfrak q$, let $\mathfrak p_x\subsetneq C_+$ be the
homogeneous prime corresponding to $x$, and choose a prime $P$
of $T$ over $\mathfrak p_x$.  The maps
$\mathcal O_{X,x}\to B_{\mathfrak q}$ and
$\mathcal O_{X,x}\to C_{\mathfrak p_x}$ are localizations of
Laurent polynomial extensions.  The completion maps are regular by
excellence \cite{Stacks}*{Tag~0AH2}.  All the local rings involved
are Cohen--Macaulay and have infinite residue field.  Hence, it follows from 
Lemmas~\resultref{lem:flat-local-multiplicity}
and~\resultref{lem:punctured-multiplicity} that
\[
 e(S_Q)=e(B_{\mathfrak q})
 =e(\mathcal O_{X,x})
 =e(C_{\mathfrak p_x})=e(T_P)\leq e(C),
\]
proving \eqref{eq:veronese-punctured-bound}.

Fix $n\geq1$, and set $E=e(A)$.  Successive general degree-$n$ sections yield a
homogeneous regular sequence $f_1,\ldots,f_t$ such that
$
 C_i=A/(f_1,\ldots,f_i) \,(0\leq i\leq t)
$
is reduced, equidimensional, and Cohen--Macaulay, with
$\Proj C_t$ smooth over $k$.  At each step, apply Bertini to the smooth locus
\cite{Stacks}*{Tag~0FD6}, and avoid the generic points of the
irreducible components of the scheme and its singular locus.
The quotient is generically reduced and Cohen--Macaulay, hence
reduced, and the dimension of the projective singular locus drops
by at least one.  Since $k$ is infinite and perfect, these choices
can be made over $k$.  This implies $t\leq d-1$.  If $A$ is normal and
$d\geq3$, then $\dim\Sing(\Proj A)\leq d-3$, so we may take
$t\leq d-2$.  When $\Proj A$ is smooth, take $t=0$.

Since each $f_i$ has degree $n$,
$
 C_i^{(n)}\cong A^{(n)}/(f_1,\ldots,f_i).
$
Their completed local rings at the homogeneous maximal ideals
therefore form an admissible flag $\mathcal F_n$ for $R_n$.
Indeed, the terminal ring has isolated singularities,
by the smoothness of $\Proj C_t$ and the regularity of the
formal fibres.  The multiplicity formulas for a homogeneous
regular sequence and a Veronese subring give
\[
 e(C_i)=n^iE,\qquad e(R_n)=n^{d-1}E.
\]
By \eqref{eq:veronese-punctured-bound} and the bound $1$ for the
empty flag,
\[
 \cinf(R_n)\leq\Lambda(\mathcal F_n)\leq
 \begin{cases}
 n^{d-2}E & \text{in general},\\
 n^{d-3}E & \text{if $A$ is normal and $d\geq3$}.
 \end{cases}
\]
The maximal ideal of $R_n$ gives
\[
 \cLM(R_n)\geq\frac{e(R_n)}{d!}
 =\frac{n^{d-1}E}{d!}.
\]
The conclusions follow by comparing these bounds and applying
Proposition~\resultref{prop:attainment-above-tail}.
If $\Sing(\Proj A)\ne\varnothing$ has codimension $c$ in $\Proj A$, the same
construction uses at most $d-c$ sections and gives
$\cinf(R_n)\leq n^{d-c-1}E$.  Comparing with the preceding lower
bound shows that $n^c>d!$ suffices.
\end{proof}

\section{Examples}
\label{sec:examples}

The following example shows that the Bertini-flag bound can be strictly
smaller than $e(R)$, and can equal $\lceil e(R)/2\rceil$.

\begin{proposition}\label{prop:balanced-hypersurface}
Let $k$ be an infinite field, let $m\geq4$, and choose
$a,b\geq2$ with $a+b=m$.  If
$\operatorname{char}k=p>0$, assume further that $p\nmid mab$.  Set
$ R=k\fps{x,y,z,w}/(x^m+y^az^b+w^m).$
Then $R$ is a complete reduced equidimensional Cohen--Macaulay ring
of dimension three and
\[
 \Sing R=V(x,w,yz)
 =V(x,y,w)\cup V(x,z,w).
\]
Moreover,
$
 e(R)=m,\, \muo(R)=\max\{a,b\},
$
and $y+z$ defines an admissible one-step flag.  Consequently,
$
 \cinf(R)\leq\max\{a,b\}.
$
In particular, if $a=\lfloor m/2\rfloor$ and
$b=\lceil m/2\rceil$, then
\[
 \cinf(R)\leq\left\lceil\frac m2\right\rceil<m=e(R).
\]
\end{proposition}

\begin{proof}
Let $S=k\fps{x,y,z,w}$ and
$f=x^m+y^az^b+w^m$.  Since $S$ is a UFD, a repeated
irreducible factor of $f$ would divide both
\[
 \frac{\partial f}{\partial x}=mx^{m-1}
 \quad\text{and}\quad
 \frac{\partial f}{\partial w}=mw^{m-1},
\]
which is impossible.  Thus $f$ is square-free and $R$ is reduced.
As a hypersurface, $R$ is Cohen--Macaulay and equidimensional of
dimension three.

Let $\mathfrak p\in\Spec R$, and let
$\widetilde{\mathfrak p}$ be its inverse image in $S$.  The
hypersurface $R_{\mathfrak p}$ is regular precisely when
$
 f\notin
 \bigl(\widetilde{\mathfrak p}S_{\widetilde{\mathfrak p}}\bigr)^2.
$
If $x,w,yz\in\widetilde{\mathfrak p}$, then every term of $f$
belongs to this square, and the localization is singular.  Conversely,
if one of $x,w,yz$ is not in $\widetilde{\mathfrak p}$, then one
of
$
 mx^{m-1},\, mw^{m-1},\, ay^{a-1}z^b
$
is a unit in $S_{\widetilde{\mathfrak p}}$.  A derivation carries
the square of the maximal ideal into that maximal ideal, so $f$
cannot belong to the square.  This proves
$
 \Sing R=V(x,w,yz).
$
The generic points of the two components of $\Sing R$ are
\[
 P=(x,y,w)R,\qquad Q=(x,z,w)R.
\]
At $P$, the element $z$ is a unit and $f$ has order $a$;
at $Q$, the element $y$ is a unit and $f$ has order $b$.
By the multiplicity formula for a hypersurface, 
$
 e(R_P)=a,\, e(R_Q)=b.
$
All other nonmaximal localizations are regular.  Hence
$\muo(R)=\max\{a,b\}$.  At the closed point, $f$ has order
$m$, so $e(R)=m$.

It remains to verify the assertion about $y+z$.
Under the substitution $y=-z$, $f$ becomes the nonzero polynomial
$
 x^m+(-1)^az^m+w^m.
$
Thus $y+z$ does not divide $f$.  Since $f$ is square-free,
$y+z$ belongs to no minimal prime of $R$ and is $R$-regular.
The quotient is
\[
 R/(y+z)\cong
 k\fps{x,z,w}/(x^m+(-1)^az^m+w^m).
\]
The same square-free argument shows that this ring is reduced.  If a
prime over the displayed equation is nonmaximal, at least one of
$x,z,w$ is a unit, and the corresponding partial derivative is a
unit.  The quotient therefore has isolated singularities.  This is an
admissible one-step flag, and Theorem~\resultref{thm:flag-bound} proves the
asserted estimate.
\end{proof}

The following example shows that the local multiplicities in Theorem~\resultref{thm:flag-bound} cannot
be replaced by $1$, even when a single section has isolated singularities.

\begin{proposition}\label{prop:ordinary-multiple-point}
Let $k$ be an infinite field, choose distinct elements
$\lambda_1,\ldots,\lambda_n\in k$, where $n\geq3$, and put
\[
 f(u,v)=\prod_{j=1}^n(u-\lambda_jv),\qquad
 A=k\fps{u,v}/(f),\qquad
 R=A\fps{z}.
\]
Then $R$ is a complete reduced Cohen--Macaulay hypersurface of
dimension two,
$
 \Sing R=V(u,v),
$
and $z$ defines an admissible one-step flag $R\to A$.  For this
flag, $\muo(R)=n$, while
$
 \frac n2\leq\cinf(R)\leq n.
$
\end{proposition}

\begin{proof}
The power series $f$ is square-free, so $A$ and $R$ are
reduced; being hypersurfaces, they are Cohen--Macaulay and
equidimensional.  Let $\widetilde{\mathfrak p}$ be a prime of
$k\fps{u,v,z}$ containing $f$.
If it does not contain both $u$ and $v$, it contains exactly one
factor $u-\lambda_jv$: two distinct factors would force both $u$
and $v$ into the prime.  All other factors are units after
localization, and $u-\lambda_jv$ does not belong to the square of
the localized maximal ideal; indeed, its $u$-derivative is one.
Hence the hypersurface localization is regular.  If
$u,v\in\widetilde{\mathfrak p}$, then
$f\in\widetilde{\mathfrak p}^{\,n}\subseteq
\widetilde{\mathfrak p}^{\,2}$, and the localization is singular.
Thus $\Sing R=V(u,v)$.

The element $z$ is $R$-regular and $R/zR=A$.  A reduced
one-dimensional local ring has isolated singularities, so this is an
admissible flag.  At the nonmaximal singular prime $P=(u,v)R$,
$
 \widehat{R_P}\cong k\lps{z}\fps{u,v}/(f).
$
Indeed, localization at $P$ inverts every nonzero element of
$k\fps{z}$.  The displayed isomorphism follows by reducing modulo
$(u,v)^r$ and taking the inverse limit.
Completion preserves multiplicity, and $f$ has order $n$ in the
two-dimensional regular local ring $k\lps{z}\fps{u,v}$.  Hence
$e(R_P)=n$.  Every other nonmaximal localization is regular, so
$\muo(R)=n$.
By Theorem~\resultref{thm:flag-bound}, $\cinf(R)\leq n$.

For the lower bound, let $\mathfrak q=(u,v)A$, and put
$
 I_L=\mathfrak qR+(z^L)\,(L\geq 1).
$
Then $R/I_L\cong k\fps{z}/(z^L)$, so
$\length_R(R/I_L)=L$.  Choose a principal minimal reduction
$(h)$ of $\mathfrak q$.  For some $N\geq2$,
$\mathfrak q^N=h\mathfrak q^{N-1}$.  Working in $R$, we obtain
\[
 I_L^N=\mathfrak q^N+z^LI_L^{N-1}
       =h\mathfrak q^{N-1}+z^LI_L^{N-1}
       =(h,z^L)I_L^{N-1},
\]
where the last equality uses
$I_L^{N-1}=\mathfrak q^{N-1}+z^LI_L^{N-2}$ and
$hI_L^{N-2}\subseteq I_L^{N-1}$.
Thus the parameter ideal $(h,z^L)$ reduces $I_L$.  The ring $R$ is
two-dimensional Cohen--Macaulay, and
$\length_A(A/hA)=e(A)=n$.  Therefore
\[
 e_R(I_L)=\length_R(R/(h,z^L))=Ln,
\]
and
\[
 \frac{e_R(I_L)}{2!\length_R(R/I_L)}=\frac n2
\]
for every $L$.  Letting $L\to\infty$ proves the lower bound.
\end{proof}

\begin{example}\label{ex:non-gorenstein-veronese}
Let $k$ be an infinite perfect field, let $r\geq3$, and put
$
 C=k[u,v]^{(r)},\, A=C[z],
$
where $C_1=k[u,v]_r$ and $\deg z=1$.  For $n\geq3$, set
\[
 R_{r,n}=\widehat{(A^{(n)})_{A^{(n)}_+}}.
\]
Then $R_{r,n}$ is a complete three-dimensional normal
Cohen--Macaulay domain, is not Gorenstein, and has a one-dimensional
singular locus.  Moreover,
$
 e(R_{r,n})=n^2r,\, \muo(R_{r,n})=r,
$
and $\cLM(R_{r,n})$ is attained and rational.

The ring $C$ is normal and free over
$k[u^r,v^r]$, with basis
$
 1,\, u^iv^{r-i}\,(1\leq i\leq r-1).
$
Thus $A$, its Veronese subrings, and their completed local rings
at the vertices are normal and Cohen--Macaulay.  The ring $C$
is regular away from its vertex, as is seen by localizing at
$u^r$ and $v^r$.  Consequently, $\Proj A$ has exactly one
singular point, namely the vertex of $C$ in the chart
$D_+(z)\cong\Spec C$.

Write $B=A^{(n)}$.  In the chart where $z^n$ is invertible,
$
 B_{z^n}\cong C[T,T^{-1}],\, T=z^n,
$
where the degree-one generators of $C$ correspond to
$u^iv^{r-i}/z$ for $0\leq i\leq r$.  Thus $\Sing B$ is the line defined by a prime
$\mathfrak q$ with $B/\mathfrak q\cong k[T]$.
Since $B$ is excellent,
\[
 \Sing R_{r,n}=V(\mathfrak qR_{r,n}),\qquad
 R_{r,n}/\mathfrak qR_{r,n}\cong k\fps{T}.
\]
The same local multiplicity comparison as in
\eqref{eq:veronese-punctured-bound} yields
$\muo(R_{r,n})=e(C)=r$, whereas
$e(R_{r,n})=n^2r$ by the Veronese multiplicity formula.

Finally, $C/(u^r,v^r)$ has a square-zero maximal ideal of
$k$-dimension $r-1$, so $C_{C_+}$ has Cohen--Macaulay type $r-1>1$.
The localization $B_{\mathfrak q}$ is obtained from $C_{C_+}$
by extending the residue field to $k(T)$ and has the same type.
Since the Gorenstein property descends under completion and is
preserved by localization, $R_{r,n}$ is not Gorenstein.
As $\dim\Sing R_{r,n}=1$,
we deduce from Theorem~\resultref{thm:flag-bound} that
\[
 \cinf(R_{r,n})\leq r<\frac{n^2r}{6}\leq\cLM(R_{r,n}).
\]
Proposition~\resultref{prop:attainment-above-tail} proves attainment
and rationality.
\end{example}

\subsection{Cones over del Pezzo products}
\label{subsec:del-pezzo-cones}

Recall that a two-dimensional rational double point $D$ satisfies
$\cLM(D)=1$ by \cite{MaSmirnov}*{Corollary~3.1.2}, and that
$\cLM(D\fps{t})\leq\cLM(D)$ by
\cite{MaSmirnov}*{Corollary~2.1.10}.
It follows that if $D$ is the local ring
of a Du Val point over an algebraically closed field $k$ of
characteristic zero, and $D\to B$ is a flat local homomorphism
with regular closed fibre, then $\cLM(B)=1$.
Indeed, write $\widehat D=k\fps{u,v,w}/(f)$, with $f$ an ADE
equation, and let $K$ be the residue field of $B$.
The map $\widehat D\to\widehat B$ is flat, and a coefficient
field $K\subset\widehat B$ can be chosen compatibly with $k$
\cite{Stacks}*{Tags~0C4G and~0C34}.
Let $a$ be the dimension of the closed fibre.
The structure map and lifts of a regular system of parameters
of the completed closed fibre define a surjection
\[
 K\fps{u,v,w,t_1,\ldots,t_a}/(f)\to\widehat B.
\]
This is an isomorphism modulo the maximal ideal of $\widehat D$;
by flatness over $\widehat D$ and Nakayama's lemma, its
kernel is zero.  Thus $\widehat B$ is a formal power series
extension of a rational double point over $K$, and the assertion
follows from the bounds above and invariance of $\cLM$ under completion.

\begin{proof}[Proof of Example~\resultref{ex:del-pezzo-cones}]
Fix $n\geq1$, and put $r=\dim Z=d-3$ and $X_n=\Spec A^{(n)}$.
The variety $Y=S\times Z$ is Gorenstein, klt, and Fano.
The cone criterion \cite{KollarSingularities}*{Lemma~3.1},
applied to $L^{\otimes n}=-nqK_Y$, shows that
$X_n$ is normal and klt; in particular, its
coordinate ring is Cohen--Macaulay.
The punctured cone is a $\mathbb G_m$-bundle over $Y$.
Since
\[
 \Sing Y=\coprod_{s\in\Sing S}\{s\}\times Z,
\]
the singular locus of the punctured cone is the inverse image
of $\Sing Y$, so that
$\dim\Sing R_n=d-2$ by excellence.

Choose $r+1$ general members $H_1,\ldots,H_{r+1}$ of
$\lvert L^{\otimes n}\rvert$, and put
$Y_i=H_1\cap\cdots\cap H_i$, with $Y_0=Y$.
By Bertini, the intersections $Y_i$ are smooth away from $\Sing Y$
and meet each of the finitely many smooth subvarieties
$\{s\}\times Z$ transversely for $i\leq r$;
the last section avoids the remaining finite set of singular
points.  These conditions can be imposed simultaneously
\cite{Stacks}*{Tag~0FD6}.
For $i\leq r$, at a point of $Y_i\cap(\{s\}\times Z)$, the differentials of
the defining equations in the $Z$-directions are independent.
Consequently, $Y_i\to S$ is smooth of relative dimension
$r-i$ near this intersection.  Thus each $Y_i$ is normal,
and $Y_{r+1}$ is smooth.

Let $h_j\in(A^{(n)})_1$ define $H_j$, and set
\[
 C_i=A^{(n)}/(h_1,\ldots,h_i),\qquad
 T_i=\widehat{(C_i)_{(C_i)_+}}
 \quad(0\leq i\leq r+1).
\]
The general linear forms $h_j$ form a regular sequence, so
each $C_i$ is Cohen--Macaulay of dimension $d-i\geq2$.
Its punctured spectrum is a $\mathbb G_m$-bundle over $Y_i$,
hence is normal.  By Serre's criterion, $C_i$ is normal
at the vertex as well.  The rings $T_i$ form an admissible Bertini
flag for $T_0=R_n$, with $T_{r+1}$ a two-dimensional ring
having isolated singularities.

For $i\leq r$, let $P$ be a nonmaximal prime of $T_i$.
If $(T_i)_P$ is not regular, its contraction to $C_i$ lies
over a point of $Y_i\cap(\{s\}\times Z)$ for some $s\in\Sing S$.
Composing the smooth maps just described with the regular
completion map, we obtain a regular local homomorphism
$
 \mathcal O_{S,s}\to(T_i)_P.
$
For any further regular local homomorphism $(T_i)_P\to B$,
where $B$ is excellent of dimension $d-i-1$,
the preceding observation therefore implies $\cLM(B)=1$.
If $(T_i)_P$ is regular, then so is $B$, and the same
conclusion holds.  Variant~\resultref{var:LM-recurrence} thus
applies with $b=1$ at every step.
Since $\cinf(T_{r+1})=1$ by Theorem~\resultref{thm:main},
we conclude by descending induction that $\cinf(R_n)=1$.

For every nonterminal ring $T_i$, $\muo(T_i)\leq2$
by the same local description and
Lemma~\resultref{lem:flat-local-multiplicity}.
Since $S$ is singular, $\muo(T_0)=2$.
The constructed flag has bound $2$, while every admissible
flag for $R_n$ includes the term $\muo(R_n)=2$.
Hence $\Lambda_{\mathrm{Bert}}(R_n)=2$.
Finally, $e(R_n)$ is the degree of $Y$ in the embedding
defined by $L^{\otimes n}$.  Taking the maximal ideal in the definition
of $\cLM$, we obtain
\[
 e(R_n)=(nq)^{d-1}(-K_Y)^{d-1},\qquad
 \cLM(R_n)\geq\frac{(nq)^{d-1}(-K_Y)^{d-1}}{d!}
 \to\infty.
\]
\end{proof}

\end{document}